\documentclass[12pt]{article}

\usepackage{amsthm}
\usepackage{amssymb}
\usepackage{amsmath}
\usepackage{graphicx}
\usepackage{mathdots}
\usepackage{mathtools}
\usepackage{url}
\usepackage{color}
\usepackage{framed}
\usepackage{tikz}
\usepackage{stmaryrd}
\usepackage[ruled,vlined]{algorithm2e}
\usepackage{array}
\usepackage{comment}

\usepackage[margin=1in]{geometry}

\newcommand{\llangle}{\langle\hspace{-2.5pt}\langle}
\newcommand{\rrangle}{\rangle\hspace{-2.5pt}\rangle}

\newtheorem{theorem}{Theorem}%[section]
\newtheorem{corollary}[theorem]{Corollary}
\newtheorem{lemma}[theorem]{Lemma}
\newtheorem{proposition}[theorem]{Proposition}

\theoremstyle{definition}

\newtheorem{example}[theorem]{Example}
\newtheorem{definition}[theorem]{Definition}

\title{Injectivity, stability, and positive definiteness\\of max filtering}

\author{Dustin~G.~Mixon\footnote{Department of Mathematics, The Ohio State University, Columbus, OH} \footnote{Translational Data Analytics Institute, The Ohio State University, Columbus, OH}
\quad
Yousef~Qaddura\footnotemark[1]
}
   
\date{}

\begin{document}
\maketitle

\begin{abstract}
Given a real inner product space $V$ and a group $G$ of linear isometries, max filtering offers a rich class of $G$-invariant maps.
In this paper, we identify nearly sharp conditions under which these maps injectively embed the orbit space $V/G$ into Euclidean space, and when $G$ is finite, we estimate the map's distortion of the quotient metric.
We also characterize when max filtering is a positive definite kernel.
\end{abstract}

\section{Introduction}

Many machine learning algorithms are designed for Euclidean data, and so it is convenient to represent objects as vectors in a real inner product space $V$.
In many cases, the naive vector representation of an object is not unique.
For example, one might represent a point cloud of $n$ points in $\mathbb{R}^d$ as a member $x$ of $V:=\mathbb{R}^{d\times n}$, but the same point cloud can also be represented by any permutation of the columns of $x$.
Observe that this ambiguity results from a subgroup $G$ of $\operatorname{O}(V)$, namely, permutations of columns.
If ignored, this ambiguity will render the machine learning process extraordinarily inefficient, requiring a much larger training set than necessary.
In practice, one might account for such ambiguities by augmenting the training set with the entire $G$-orbit of each example~\cite{SimardSP:03,CiresanMGS:10,KrizhevskySH:12,ChenDL:20}.
When $G$ is large, this makes the machine learning process much more computationally intensive than necessary.

As an alternative, one might factor out the ambiguity by identifying objects with members $[x]:=G\cdot x$ of the orbit space $V/G$.
This orbit space is endowed with the quotient metric
\[
d([x],[y]):=\inf_{\substack{p\in[x]\\q\in[y]}}\|x-y\|
\]
provided the orbits of $G$ are closed.
In order to make use of the vast array of Euclidean-based machine learning algorithms, we are inclined to embed the orbit space into Euclidean space while simultaneously minimizing the resulting distortion of the quotient metric.

\begin{definition}

%blue
\textcolor{black}{Given a map $f\colon V/G\to \mathbb R^n$, we may take $\alpha,\beta\in[0,\infty]$ to be the largest and smallest constants (respectively) such that}
\[\alpha \cdot d([x],[y]) \leq \|f([x])-f([y])\|\leq \beta\cdot d([x],[y]) \qquad \forall x,y\in V.\]
%blue
\textcolor{black}{Then, the \textbf{distortion} of $f$ is given by $\kappa(f):= \frac{\beta}{\alpha}$, where we take $\kappa(f)=\infty$ if $\alpha=0$ or $\beta = \infty$.}

\end{definition}

We call $\alpha$ and $\beta$ the \textbf{optimal lower and upper Lipschitz bounds} respectively. Section~2 of~\cite{CahillIM:24} shows how minimizing distortion is particularly useful in the context of transferring Euclidean data science algorithms (e.g., nearest neighbors, clustering and multidimensional scaling algorithms) to arbitrary metric spaces. To illustrate, we adapt an example to the context of nearest neighbor search over orbit spaces.

\begin{example}[Example~1 in~\cite{CahillIM:24}] 
  \label{ex.bilip nearest}
  Given $\lambda\geq 1$ and data $x_1,\dots,x_m\in V/G$, the \textbf{$\lambda$-approximate nearest neighbor problem} takes as input $[x]\in V/G$ and outputs $j\in \{1,\dots,m\}$ such that
  \[d([x],[x_j])\leq \lambda\cdot \min_{1\leq i\leq m} d([x],[x_i]).\]
  Given a map $f\colon V/G\to \mathbb R^n$ with distortion $c$ and a black box algorithm that solves the $\lambda$-approximate nearest neighbor problem in $\mathbb R^n$, one may transfer the algorithm to $V/G$ by pulling back through $f$. This results in a solution of the $c\lambda$-approximate nearest neighbor problem in $V/G$. To see this, first use the black box algorithm to find $j\in \{1,\dots,m\}$ such that
  \[\|f([x])-f([x_j])\|\leq \lambda\cdot \min_{1\leq i\leq m} \|f([x])-f([x_i])\|.\]
  Then 
  \[d([x],[x_j])\leq\frac{1}{\alpha}\cdot \|f([x])-f([x_j])\|
  \leq\frac{\lambda}{\alpha}\cdot\min_{i\in I}\|f([x])-f([x_i])\|
  \leq c\lambda\cdot\min_{i\in I} d([x],[x_i]).
  \]
\end{example}

\begin{comment}
Before %blue
\textcolor{black}{REMOVE this paragraph?: the interest in the theory of bilipschitz invariants grew}, the bulk of the literature focused on polynomial invariants~\cite{BandeiraBKPWW:17,PerryWBRS:19,CahillCC:20,BendoryELS:22,BalanHS:22}.
However, every Lipschitz polynomial map $\mathbb{R}^d/G\to\mathbb{R}^n$ is affine linear, which fails to be injective when $G$ is nontrivial.
As such, a bilipschitz embedding of the orbit space required a non-polynomial approach.
%blue
\textcolor{black}{ Notably, \cite{Olver:19} is a precursor to the non-polynomial approach of max filtering.}
\end{comment}

As such, we seek bilipschitz embeddings of orbit spaces into Euclidean spaces. To this end, \cite{CahillIMP:22} recently introduced a family of embeddings called \textit{max filter banks} that enjoy explicit bilipschitz bounds whenever $G$ is finite.

\begin{definition}
Consider any real inner product space $V$ and $G\leq\operatorname{O}(V)$.
\begin{itemize}
\item[(a)]
The \textbf{max filtering map} $\llangle\cdot,\cdot\rrangle\colon V/G\times V/G\to\mathbb{R}$ is defined by
\[
\llangle [x],[y]\rrangle
:=\sup_{\substack{p\in[x]\\q\in[y]}}\langle p,q\rangle.
\]
\item[(b)]
Given \textbf{templates} $z_1,\ldots,z_n\in V$, the corresponding \textbf{max filter bank} $\Phi\colon V/G\to\mathbb{R}^n$ is defined by
\[
\Phi([x])
:=\{\llangle [z_i],[x]\rrangle\}_{i=1}^n.
\]
\end{itemize}
\end{definition}

%blue
\textcolor{black}{In loose terms, an individual max filter map $\llangle [z_i],\cdot \rrangle$ is a scalar feature map which takes as input $[x]\in V/G$ and measures the maximal alignment between the orbits $[x]$ and $[z_i]$ when viewed as subsets of $V$. A max filter bank is then a collection of such feature maps. While we focus on estimating their distortions, we give two crucial remarks on the practical aspects of max filter banks.}

%blue
\textcolor{black}{Firstly, by virtue of the polarization identity (Lemma~2(f) in~\cite{CahillIMP:22}), we have
\[d^2([x],[y]) = \|x\|^2 + \|y\|^2 - 2\llangle [x],[y]\rrangle.\]
Thus, the max filtering map and the quotient distance have the same computational complexity. Fortunately, in Section~3.2 of~\cite{CahillIMP:22}, the max filtering map was shown to be easy to compute for many important instances of $G$. For example:}
\begin{example}[Section~3.2.2 in~\cite{CahillIMP:22}] When $V\cong \mathbb R^n$ and $G\cong \mathbb Z/n\mathbb Z$ is the group of circular shifts, the max filter map is given in terms of a convolution operation:
\[\llangle [f],[g]\rrangle = \max_{a\in Z/n\mathbb Z} \sum_{x\in \mathbb Z/n\mathbb Z} f(x)g(x-a),\]
which can be computed in linearithmic time with the aid of the fast Fourier transform.
\end{example}

%blue
\textcolor{black}{Second, from a machine learning perspective, we note that the max filter map $\llangle [x],\cdot\rrangle\colon V\to \mathbb R$ is convex with an explicit subgradient (Section~5.3 in~\cite{CahillIMP:22}), and so one may include a max filter bank as a $G$-invariant layer in a classification model where constituent templates are trainable parameters. We refer the reader to Section~6 in~\cite{CahillIMP:22} for relevant numerical examples.}

%blue
\textcolor{black}{In this paper, we estimate the bilipschitz bounds and distortion of max filter banks. We identify conditions under which max filter banks are injective and bilipschitz, and we characterize when the max filtering map is positive definite.}

\subsection{Outline and relevance of results}

The following is a summary of our results and their relevance in comparison with previous literature.

\begin{itemize}

\item In Section~\ref{sec.injectivity}, we show that for every $G\leq\operatorname{O}(d)$ with closed orbits, $n\geq 2d$ generic templates suffice for the corresponding max filter bank to be injective (Theorem~\ref{thm.injectivity}). This improves over Corollary~13 in~\cite{CahillIMP:22}, which is the same result, but requires $G$ to be finite. This is relevant to the study of bilipschitz maps since they are necessarily injective by virtue of $\alpha>0$.

\item In Section~\ref{sec.upper stability}, we estimate the upper Lipschitz bound for every $G\leq\operatorname{O}(d)$ with closed orbits (Theorems~\ref{thm.upper lipschitz bound finite} and \ref{thm.upper lipschitz bound}). The estimates are relevant in two ways. They are optimal when $G$ is finite or $-\operatorname{id}\in G$, and they depend only on the templates $z_1,\dots,z_n\in \mathbb R^d$ along with their orbits.

\item Section~\ref{sec.lower stability} has two subsections. In the first subsection, we derive a lower Lipschitz bound $\alpha$ for all finite groups $G$ by leveraging the geometry of the Voronoi cell decomposition afforded by $G$ (Theorem~\ref{thm.lower lipschitz bound}). Notably, the bound is optimal for many examples of $G$ (Lemma~\ref{lem.lower lipschitz optimal cases}). In the second subsection, we introduce a quantity $\chi(G)$ called the Voronoi characteristic of $G$ (Definition~\ref{def.voronoi characteristic}) satisfying $\chi(G)\leq |G|$, and we use it to derive a more theoretically accessible lower Lipschitz bound $\tilde \alpha$ (Corollary~\ref{cor.weaker bound}). We show that $n\geq \chi\cdot (d-1) +1 $ generic templates suffice for max filter banks to be bilipschitz (Theorem~\ref{thm.generic bilipschitz}). Lastly, we evaluate our estimates on random templates with standard Gaussian entries. We obtain a distortion estimate when $G$ is finite (Theorem~\ref{thm.distortion from random templates}). We shall find that for any finite $G\leq \operatorname{O}(d)$ of order $m$ and any $\epsilon>0$, there exists $C_1 > 0$ such that for $z_1,\dots, z_n\in \mathbb R^d$ with $\mathrm{iid}$ standard Gaussian entries, it holds that
\[\kappa(\Phi)\leq C_1\chi^{3/2+\epsilon}m^{1+\epsilon}
\leq C_1m^{5/2+2\epsilon}
\]
with high probability in $n$. Here, $\Phi\colon\mathbb{R}^d/G\to\mathbb{R}^n$ denotes the max filter bank defined by $\Phi([x]):=\{\llangle [z_i],[x]\rrangle\}_{i=1}^n$ and $\chi\leq m$ denotes the Voronoi characteristic of $G$ (Definition~\ref{def.voronoi characteristic}). This is relevant in two ways. First, the bound is sharp up to $O(m^{\epsilon})$ factors when $G\leq \operatorname{O}(2)$ is the dihedral group (see Theorem~7(d) in~\cite{MixonP:22}). Second, our results improve over Theorem~18 in~\cite{CahillIMP:22}, which reports suboptimal bilipschitz bounds for templates drawn uniformly from the unit sphere. More specifically, by optimizing their analysis, one obtains the bound
\[
\kappa(\Phi)
\leq C_0m^3d^{1/2}(d\log d+d\log m+\log^2m)^{1/2},
\]
which depends on $d$ and has suboptimal exponent over $m$.

\item In Section~\ref{sec.pos def}, we characterize the compact groups $G\leq\operatorname{O}(d)$ for which the max filtering map $\llangle\cdot,\cdot\rrangle$ is a positive definite kernel (Theorem~\ref{thm.max filtering pos def}). This takes place if and only if the orbit space of $G$ is isometric to a subset of Euclidean space, i.e.,\ there exists an embedding $f\colon V/G\to \mathbb R^n$ with $\kappa(f)=1$. Note that this is the best possible distortion one may ask for. As such, all other orbit spaces must incur nontrivial distortion when embedded into Euclidean space. We shall find that positive definitness of the max filtering map renders it equivalent to a max filtering map with an appropriate finite reflection group. We refer the reader to~\cite{MixonP:22} for a deeper distortion analysis in the setting of general finite reflection groups and to Section~5.4 and Example~32 in~\cite{CahillIMP:22} for a numerical analysis and application in the setting of permutation groups.

\end{itemize}

\section{Injectivity}
\label{sec.injectivity}

In this section, we show that $2d$ generic templates suffice for any max filter bank to be injective.
More precisely, what follows is the main result of this section:

\begin{theorem}
\label{thm.injectivity}
Suppose $G\leq\operatorname{O}(d)$.
\begin{itemize}
\item[(a)]
If any orbit of $G$ is not closed, then every continuous map from $\mathbb{R}^d/G$ to any $T_1$ space is not injective.
\item[(b)]
If the orbits of $G$ are concentric spheres in $\mathbb{R}^d$, then for any nonzero $z\in\mathbb{R}^d$, the max filter $[x]\mapsto\llangle [z],[x]\rrangle$ is injective.
\item[(c)]
If the orbits of $G$ are closed, then for generic $z_1,\ldots,z_n\in\mathbb{R}^d$, the max filter bank $[x]\mapsto\{\llangle [z_i],[x]\rrangle\}_{i=1}^n$ is injective provided $n\geq 2d$.
\end{itemize}
\end{theorem}

Theorem~\ref{thm.injectivity}(c) generalizes Corollary~13 in~\cite{CahillIMP:22}, which only treats the case where $G$ is finite.
The proofs of (a) and (b) are straightforward, while our proof of (c) makes use of some technology from~\cite{CahillIMP:22}.
First, we review some facts from real algebraic geometry; see~\cite{BochnakCR:13} for more details.

A basic semialgebraic set is any set of the form $\{x\in\mathbb{R}^n:p(x)\geq0\}$, where $p\colon\mathbb{R}^n\to\mathbb{R}$ is a polynomial function.
A \textbf{semialgebraic} set is any set obtained from some combination of finite unions, finite intersections, and complements of basic semialgebraic sets.
We say a subgroup of $\operatorname{GL}(d)$ is semialgebraic if it is semialgebraic as a subset of $\mathbb{R}^{d\times d}$.
We say a function $\mathbb{R}^s\to\mathbb{R}^t$ is semialgebraic if its graph is semialgebraic as a subset of $\mathbb{R}^{s+t}$.

\begin{proposition}[Lemma~10 in~\cite{CahillIMP:22}]
\label{prop.max filtering is semialgebraic}
For every semialgebraic subgroup $G\leq\operatorname{O}(d)$, the corresponding max filtering map $\llangle[\cdot],[\cdot]\rrangle\colon\mathbb{R}^d\times\mathbb{R}^d\to\mathbb{R}$ is semialgebraic.
\end{proposition}

Every semialgebraic set can be expressed as a disjoint union of finitely many semialgebraic sets such that each is either homeomorphic to an open hypercube of positive dimension or a singleton set of zero dimension.
The \textbf{dimension} of a semialgebraic set is the maximum of these dimensions over the decomposition, though it does not depend on the choice of decomposition.

\begin{definition}[cf.\ \cite{DymG:22,CahillIMP:22}]
Given a semialgebraic subgroup $G\leq\operatorname{O}(d)$, we say the corresponding max filtering map $\llangle [\cdot],[\cdot]\rrangle\colon\mathbb{R}^d\times\mathbb{R}^d\to\mathbb{R}$ is $k$-\textbf{strongly separating} if for every $x,y\in\mathbb{R}^d$ with $[x]\neq[y]$, it holds that
\[
\operatorname{dim}\big\{z\in\mathbb{R}^d:\llangle [z],[x]\rrangle=\llangle [z],[y]\rrangle\big\}
\leq d-k.
\]
\end{definition}

\begin{proposition}[Theorem~12 in~\cite{CahillIMP:22}]
\label{prop.semialgebraic injective}
Consider any semialgebraic subgroup $G\leq\operatorname{O}(d)$ with $k$-strongly separating max filtering map $\llangle[\cdot],[\cdot]\rrangle\colon\mathbb{R}^d\times\mathbb{R}^d\to\mathbb{R}$ for some $k\in\mathbb{N}$.
For generic $z_1,\ldots,z_n\in\mathbb{R}^d$, the max filter bank $[x]\mapsto\{\llangle[z_i],[x]\rrangle\}_{i=1}^n$ is injective provided $n\geq 2d/k$.
\end{proposition}

We will use Proposition~\ref{prop.semialgebraic injective} to prove Theorem~\ref{thm.injectivity}(c).
To prove strongly separating, we will leverage the differentiability of max filtering:

\begin{proposition}
\label{prop.max filtering derivative}
Suppose $G\leq\operatorname{O}(d)$ is closed.
\begin{itemize}
\item[(a)]
$\llangle [\cdot],[x]\rrangle$ is differentiable at $z\in\mathbb{R}^d$ if and only if $[x]$ has a unique closest point $y$ to $z$, in which case $\nabla\llangle [\cdot],[x]\rrangle(z)=y$.
\item[(b)]
For each $[x]\in\mathbb{R}^d/G$, there is an open and dense subset $U_x$ of $\mathbb{R}^d$ such that $\llangle [\cdot],[x]\rrangle$ is differentiable at every $z\in U_x$.
\item[(c)]
For any $x,z\in\mathbb{R}^d$, the open interval between $z$ and any of its closest points in $[x]$ is contained in $U_x$. 
\end{itemize}
\end{proposition}

\begin{proof}
First, (a) follows from Lemma~2(f) and Theorem~27 in~\cite{CahillIMP:22}, since this characterizes when the subgradient of $\llangle [\cdot],[x]\rrangle$ is singleton.
For (b), Proposition~1.20 in~\cite{Meinrenken:03} gives that $[x]$ is a smooth, compact submanifold of $\mathbb{R}^d$, and so Corollary~3.14 in~\cite{DudekH:94} delivers an open and dense subset $U_x$ of $\mathbb{R}^d$ such that every $z\in U_x$ has a unique closest point to $[x]$; the result then follows from the characterization in (a).
Finally, (c) follows from Theorem~3.13a in~\cite{DudekH:94}.
\end{proof}

We will also leverage the following (useful) equivalence:

\begin{proposition}
\label{prop.closed subgroups}
Suppose $G\leq\operatorname{O}(d)$.
If the orbits of $G$ are closed, then they are also the orbits of the topological closure $\overline{G}$ of $G$ in $\operatorname{O}(d)$.
Furthermore, the following are equivalent:
\begin{itemize}
\item[(a)]
$G$ is topologically closed.
\item[(b)]
$G$ is algebraic.
\item[(c)]
$G$ is semialgebraic.
\end{itemize}
\end{proposition}

\begin{proof}
Since $\operatorname{O}(d)$ is compact, the first claim follows from Theorem~5 in~\cite{DiazRamos:08}.
Next, Theorem~3.4.5 in~\cite{OnishchikV:12} gives (a)$\Rightarrow$(b), while (b)$\Rightarrow$(c) is immediate.
For (c)$\Rightarrow$(a), observe that multiplication and inversion over $\operatorname{O}(d)$ are algebraic morphisms.
When restricting to $G$, they become semialgebraic.
As such, Proposition~3.2 in~\cite{ChoiS:05} gives the result.
\end{proof}

We are now ready to prove the main result.

\begin{proof}[Proof of Theorem~\ref{thm.injectivity}]
For (a), let $Y$ be a $T_1$ space, and let $f\colon\mathbb{R}^d/G\to Y$ be continuous.
By assumption, there is a point $[x]\in \mathbb{R}^d/G$ such that $\{[x]\}$ is not closed, but since $Y$ is $T_1$, the image $\{f([x])\}$ is closed.
By continuity, it follows that the inverse image $f^{-1}(\{f([x])\})$ is closed, and is therefore a strict superset of $\{[x]\}$.
Thus, $f$ is not injective.

For (b), observe that $\llangle [z],[x]\rrangle = \|z\|\|x\|$, which determines $[x]$ whenever $z\neq0$.

For (c), considering Propositions~\ref{prop.closed subgroups} and~\ref{prop.semialgebraic injective}, it suffices to show that the max filtering map over $\overline{G}$ is $1$-strongly separating, i.e., the dimension of the semialgebraic set
\[
S([x],[y])
:=\big\{z\in\mathbb{R}^d:\llangle [z],[x]\rrangle=\llangle [z],[y]\rrangle\big\}
\]
is at most $d-1$ for each $[x]\neq[y]$.
Suppose $\operatorname{dim}S([x],[y])=d$ for some $[x],[y]\in\mathbb{R}^d/\overline{G}$.
Then $S([x],[y])$ contains a nonempty open subset $U$ of $\mathbb{R}^d$.
By Proposition~\ref{prop.max filtering derivative}(b), $\llangle [\cdot],[x]\rrangle$ and $\llangle [\cdot],[y]\rrangle$ are differentiable on open and dense subsets $U_x$ and $U_y$ of $\mathbb{R}^d$, respectively.
Then $U\cap U_x\cap U_y$ is a nonempty open set, over which we have an equality of differentiable functions:
\[
\llangle [z],[x]\rrangle
=\llangle [z],[y]\rrangle
\qquad \forall z\in U\cap U_x\cap U_y.
\]
By Proposition~\ref{prop.max filtering derivative}(a), taking the gradient of both sides at any such $z$ gives $gx=g'y$ for some $g,g'\in \overline{G}$.
Overall, $\operatorname{dim}S([x],[y])=d$ implies $[x]=[y]$, as desired.
\end{proof}

In what follows, we demonstrate that Theorem~\ref{thm.injectivity} is the best-possible result that factors through Proposition~\ref{prop.semialgebraic injective}, thereby resolving Problem~15(a) in~\cite{CahillIMP:22}.
Despite this, Theorem~\ref{thm.injectivity} is still suboptimal in settings like real phase retrieval~\cite{BalanCE:06} (where $G=\{\pm\operatorname{id}\}$ and $2d-1$ generic templates suffice) and reflection groups~\cite{MixonP:22} (where $d$ generic templates suffice).
We discuss the latter setting in more detail in Section~\ref{sec.pos def}, and we leave further refinements of Theorem~\ref{thm.injectivity} for future work.

\begin{theorem}
Given a semialgebraic subgroup $G\leq\operatorname{O}(d)$, the corresponding max filtering map is $d$-strongly separating if the orbits of $G$ are concentric spheres in $\mathbb{R}^d$.
Otherwise, the max filtering map is $1$-strongly separating, but not $2$-strongly separating.
\end{theorem}

\begin{proof}
First, suppose the orbits of $G$ are concentric spheres in $\mathbb{R}^d$.
Then $[x]\neq[y]$ implies
\[
\big\{z\in\mathbb{R}^d:\llangle [z],[x]\rrangle=\llangle [z],[y]\rrangle\big\}
=\big\{z\in\mathbb{R}^d:\|z\|\cdot\big(\|x\|-\|y\|\big)=0\big\}
=\{0\},
\]
which has dimension zero.

Otherwise, there exist $x$ and $y$ in the unit sphere in $\mathbb{R}^d$ such that $[x]\neq[y]$.
The proof of Theorem~\ref{thm.injectivity}(c) gives that the max filtering map is $1$-strongly separating.
It remains to show that it is not $2$-strongly separating.
By Proposition~\ref{prop.closed subgroups}, $G$ is topologically closed, and so $[x]$ and $[y]$ are compact.
Thus, we may assume $\llangle[x],[y]\rrangle=\langle x,y\rangle$ without loss of generality.
As such, Lemma~2(f) in~\cite{CahillIMP:22} gives that $x$ minimizes distance from $y$ over $[x]$, and similarly, $y$ minimizes distance from $x$ over $[y]$.
Consider $z\colon[0,1]\to\mathbb{R}^d$ defined by $z(t):=(1-t)x+ty$ and $h\colon\mathbb{R}^d\to\mathbb{R}$ defined by $h(z):=\llangle [z],[x]\rrangle-\llangle [z],[y]\rrangle$.
Then $(h\circ z)(0)=1-\langle x,y\rangle>0$ and $(h\circ z)(1)=\langle x,y\rangle-1<0$; here, the inequalities follow from the fact that $x$ and $y$ are distinct points on the unit sphere.
By the intermediate value theorem, there exists $t_0\in(0,1)$ such that $(h\circ z)(t_0)=0$.

By Proposition~\ref{prop.max filtering derivative}(c), we have $z(t)\in U_x\cap U_y$ for every $t\in(0,1)$.
Furthermore, by Proposition~\ref{prop.max filtering derivative}(b), $h$ is differentiable over $U_x\cap U_y$.
For each $t\in(0,1)$, the reverse triangle inequality gives that $x$ minimizes distance from $z(t)$ over $[x]$, and that $y$ minimizes distance from $z(t)$ over $[y]$.
Then Proposition~\ref{prop.max filtering derivative}(a) implies 
\begin{align*}
(h\circ z)'(t)
&=\frac{d}{dt}\Big(\llangle [z(t)],[x]\rrangle-\llangle [z(t)],[y]\rrangle\Big)\\
&=\nabla\llangle [\cdot],[x]\rrangle(z(t))\cdot(y-x)-\nabla\llangle [\cdot],[y]\rrangle(z(t))\cdot(y-x)
=-\|x-y\|^2,
\end{align*}
which is nonzero by assumption.
Considering $(h\circ z)'(t)=\nabla h(z(t))\cdot z'(t)$, it follows that $\nabla h$ is nonzero at $z(t_0)$.
The implicit function theorem then gives that $h$ vanishes on a $(d-1)$-dimensional submanifold in a neighborhood of $z(t_0)$, which implies the result.
\end{proof}

\begin{comment}
\section*{Stability}
\label{sec.stability}

In the previous section, we identified conditions under which a max filter bank is injective.
In practice, it is important to consider the extent to which the quotient metric on $\mathbb{R}^d/G$ is distorted by this embedding $\Phi$ into Euclidean space $\mathbb{R}^n$.
In this section, we derive upper and lower Lipschitz bounds for this embedding.
Explicitly, we seek $0<\alpha\leq\beta<\infty$ such that
\[
\alpha \cdot d([x],[y])
\leq \|\Phi([x])-\Phi([y])\|
\leq \beta \cdot d([x],[y])
\]
for all $[x],[y]\in \mathbb{R}^d/G$.
We identify the optimal value of $\beta$ for all such $G$, we derive a value of $\alpha$ for all finite groups $G$ that happens to be optimal for many choices of $G$, and we evaluate our bounds in the settings where the templates $\{z_i\}_{i=1}^n$ are generic or Gaussian.
Throughout this section, it is convenient to conflate a sequence $\{z_i\}_{i=1}^n$ in $\mathbb{R}^d$ with the $d\times n$ matrix whose $i$th column vector is $z_i$.

\end{comment}

\section{Upper Stability}
\label{sec.upper stability}

In the previous section, we identified conditions under which a max filter bank is injective.
In practice, it is important to consider the extent to which the quotient metric on $\mathbb{R}^d/G$ is distorted by this embedding $\Phi$ into Euclidean space $\mathbb{R}^n$.
In this section, we estimate the upper Lipschitz bound for this embedding and prove optimality of estimates when $G$ is finite or $-\operatorname{id}\in G$. Throughout this section, it is convenient to conflate a sequence $\{z_i\}_{i=1}^n$ in $\mathbb{R}^d$ with the $d\times n$ matrix whose $i$th column vector is $z_i$.

\subsection{Finite Groups}
\textbf{Throughout this subsection, we assume $G$ is finite.}

We give the optimal upper Lipschitz bound for any max filter bank. Our approach resembles a talk given by Dorsa Goreishi at the AMS 2024 Spring Southeastern Sectional meeting titled ``Generalizing max filtering with principal angles,'' which sketched a proof of the optimal upper Lipschitz bound of any $G$-invariant and positively homogeneous piecewise linear map in terms of its conical decomposition. In this section, we obtain the optimal upper Lipschitz bound by leveraging the notion of Voronoi decompositions which we now introduce.

%\textcolor{red}{INSPIRATION, EASIER PROOF THAN MY PAPER, HAS EQUALITY, ILLUSTRATES IMPORTANCE OF VORONOI, SHOWS FAILURE OF COMPACT GROUPS, COMPACT GROUPS GUY THOUGH IS MORE IMPORTANT FOR DISTORTION}.

Let $V_x$ denote the \textbf{open Voronoi cell} of $x$ relative to $[x]$.
This set is characterized by
\begin{equation}
\label{eq.open voronoi characterization}
y\in V_x
\quad
\Longleftrightarrow
\quad
\{x\}=\arg\min_{p\in[x]}\|p-y\|
\quad
\Longleftrightarrow
\quad
\{x\}=\arg\max_{p\in[x]}\langle p,y\rangle.
\end{equation}
Notice that the topological closure is similarly characterized by
\begin{alignat}{5}
y\in \overline{V_x}
\nonumber
&\quad
\Longleftrightarrow
\quad
x\in\arg\min_{p\in[x]}\|p-y\|
\quad
&&\Longleftrightarrow
\quad
x\in\arg\max_{p\in[x]}\langle p,y\rangle&&\\
\label{eq.closed voronoi characterization}
&\quad
\Longleftrightarrow
\quad
d([x],[y])=\|x-y\|
\quad
&&\Longleftrightarrow
\quad
\llangle [x],[y]\rrangle=\langle x,y\rangle&&
\quad
\Longleftrightarrow
\quad
x\in\overline{V_y}.
\end{alignat}
Next, the \textbf{Voronoi diagram} $Q_x:=\bigsqcup_{p\in[x]}V_p$ is a $G$-invariant dense disjoint union of open polyhedral cones with
\begin{equation}
\label{eq.disjoint union of cones}
y\in Q_x
\quad
\Longleftrightarrow
\quad
|\arg\max_{p\in[x]}\langle p,y\rangle|=1.
\end{equation}

For the following two examples, see Figure~\ref{fig:voronoi.first} for an illustration.

\begin{figure}[t]
  \centering
  \includegraphics[scale=0.55]{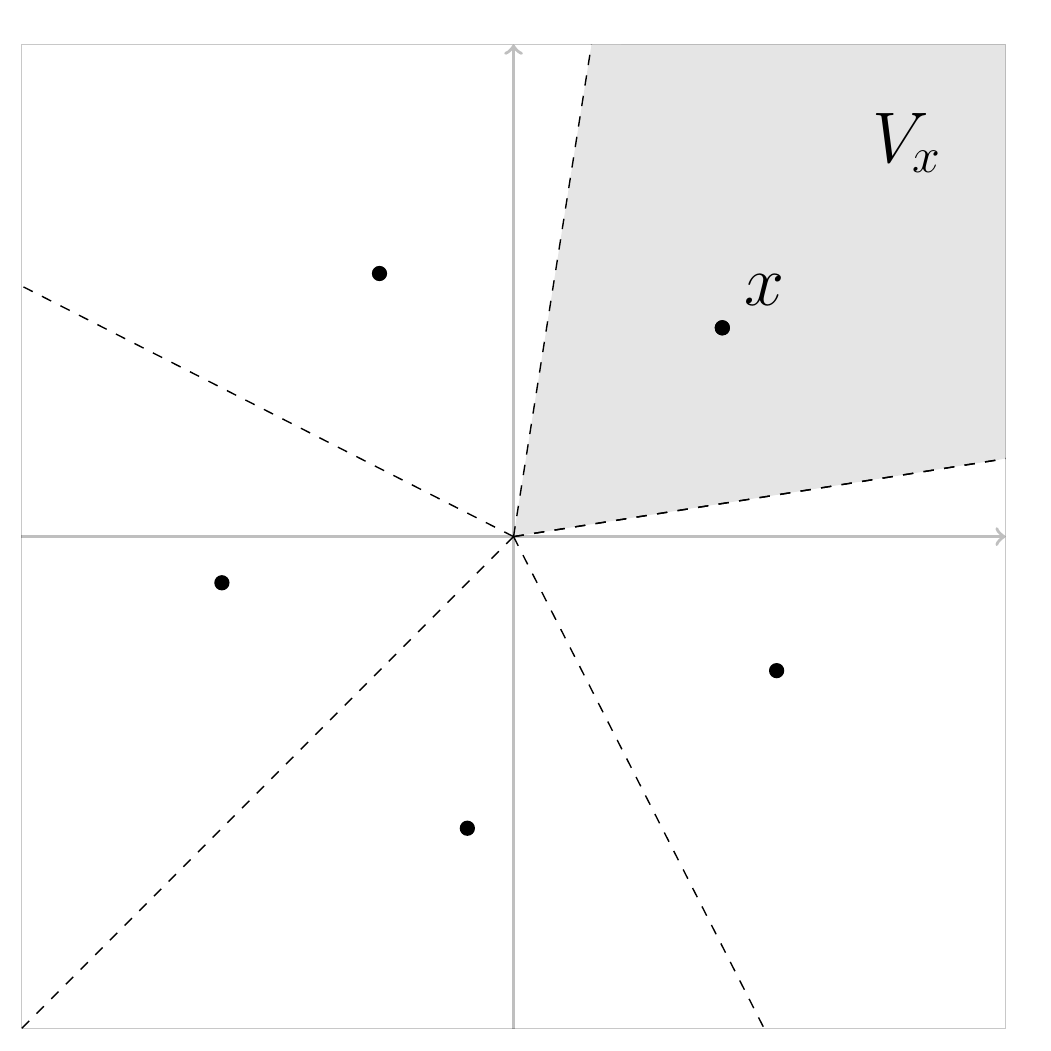}
  \includegraphics[scale=0.55]{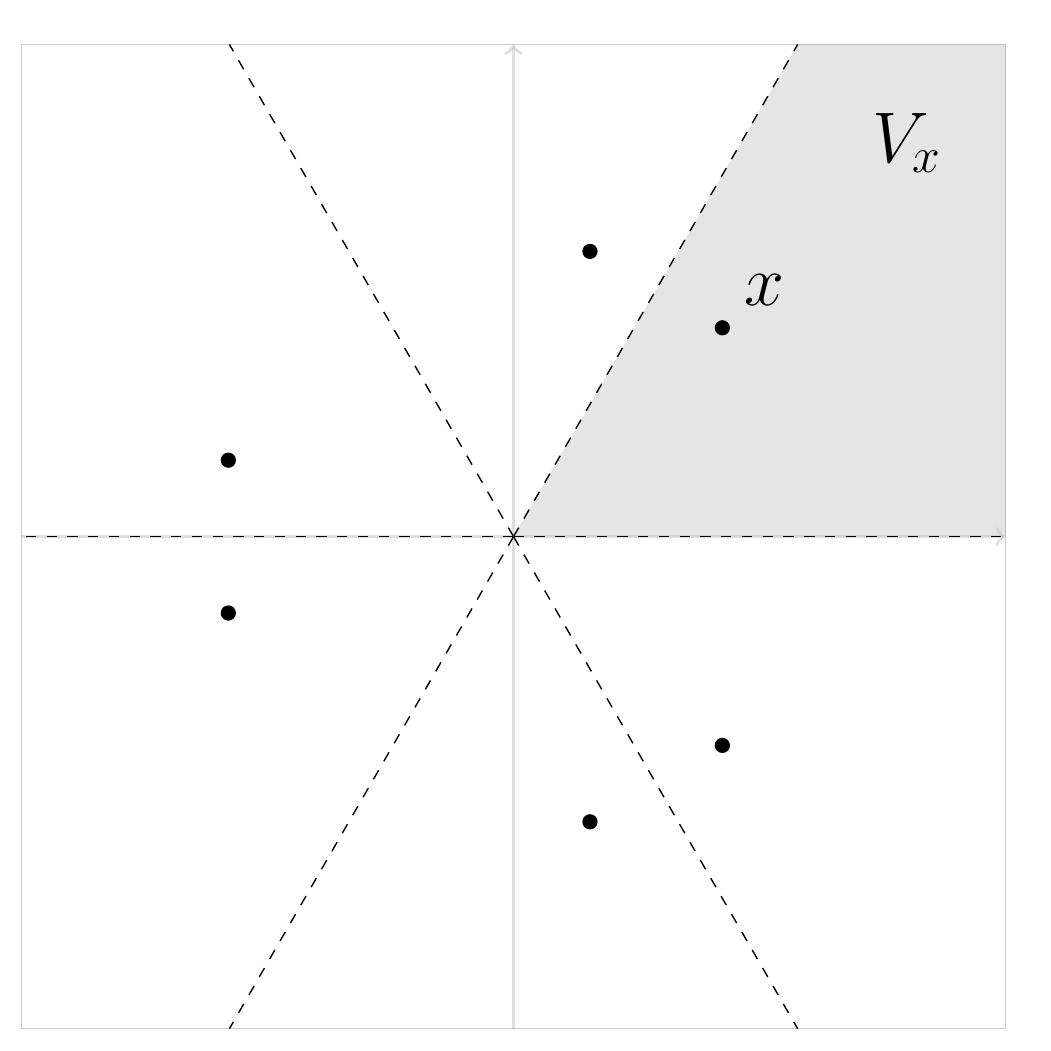}
  \caption{
    \textbf{(left)} Illustration of Example~\ref{ex.voronoi 2d rot} when $n=5$. The resulting Voronoi diagram $Q_x$ is sensitive to the argument of $x$. Note that $Q_x$ is the complement of the dashed lines. \textbf{(right)} Illustration of Example~\ref{ex.voronoi reflection} when $G$ is generated by reflections across the dashed lines. The resulting Voronoi diagram $Q_x$ does not change when $x$ is perturbed to another point in $V_x$.
  }
  \label{fig:voronoi.first}
\end{figure}
\begin{example}
  \label{ex.voronoi 2d rot}
Take $V=\mathbb R^2$ and suppose that $G\leq \operatorname{O}(2)$ consists of rotations by multiples of $\frac{2\pi}{n}$ radians on $V=\mathbb R^2$. For nonzero $x\in V$, it holds that
\[V_x = \big\{ p\in \mathbb R^2\setminus \{0\}: \arg p \in (\arg x - \tfrac{\pi}{n}, \arg x + \tfrac{\pi}{n}) \bmod 2\pi\big\},\]
while $V_0 = V$.
\end{example}

\begin{example}
  \label{ex.voronoi reflection}
Suppose $G\leq \operatorname O(V)$ is finite and generated by reflections across hyperplanes. For $x\in V$, let $\mathcal W_x$ denote the collection of closed Weyl chambers that contain $x$. Then $\overline V_x = \cup \mathcal W_x$. In particular, given an open Weyl chamber $W$, it holds that $V_x = W$ for all $x\in W$.
\end{example}

\textcolor{black}{Given templates $z_1,\dots, z_n$, it holds that $K\subseteq \mathbb R^d$ is a connected component of $\cap_{i=1}^n Q_{z_i}$ if and only if there exist $g_1,\dots,g_n\in G$ such that $ K = \cap_i V_{g_iz_i}\neq \varnothing$. In this case, $g_1z_1,\dots, g_nz_n\in \mathbb R^d$ are in fact unique, and we say that $\{g_iz_i\}_{i=1}^n$ is the $d\times n$ \textbf{matrix associated }to $K$. The following result shows that the optimal upper Lipschitz bound is given by the largest possible matrix operator $2$-norm $\|\cdot\|_{2\to 2}$ one may obtain from transformed templates $\{g_iz_i\}_{i=1}^n$ associated to connected components of $\cap_{i=1}^n Q_{z_i}$.}

\begin{theorem}
\label{thm.upper lipschitz bound finite}
Given finite $G\leq\operatorname{O}(d)$ with orbits $z_1,\ldots,z_n\in\mathbb{R}^d$, the max filter bank $\Phi\colon\mathbb{R}^d/G\to\mathbb{R}^n$ defined by $\Phi([x]):=\{\llangle [z_i],[x]\rrangle\}_{i=1}^n$ satisfies
\begin{equation}
  \label{eq.optimal upper bound finite}
\sup_{\substack{[x],[y]\in\mathbb{R}^d/G\\{[x]\neq[y]}}}\frac{\|\Phi([x])-\Phi([y])\|}{d([x],[y])}
=\max_{\substack{g_1,\ldots,g_n\in G\\ \cap_i V_{g_iz_i}\neq \varnothing}}\|\{g_iz_i\}_{i=1}^n\|_{2\to2}.
\end{equation}
\end{theorem}
\begin{figure}[t]
  \centering
  \includegraphics[scale = 0.75]{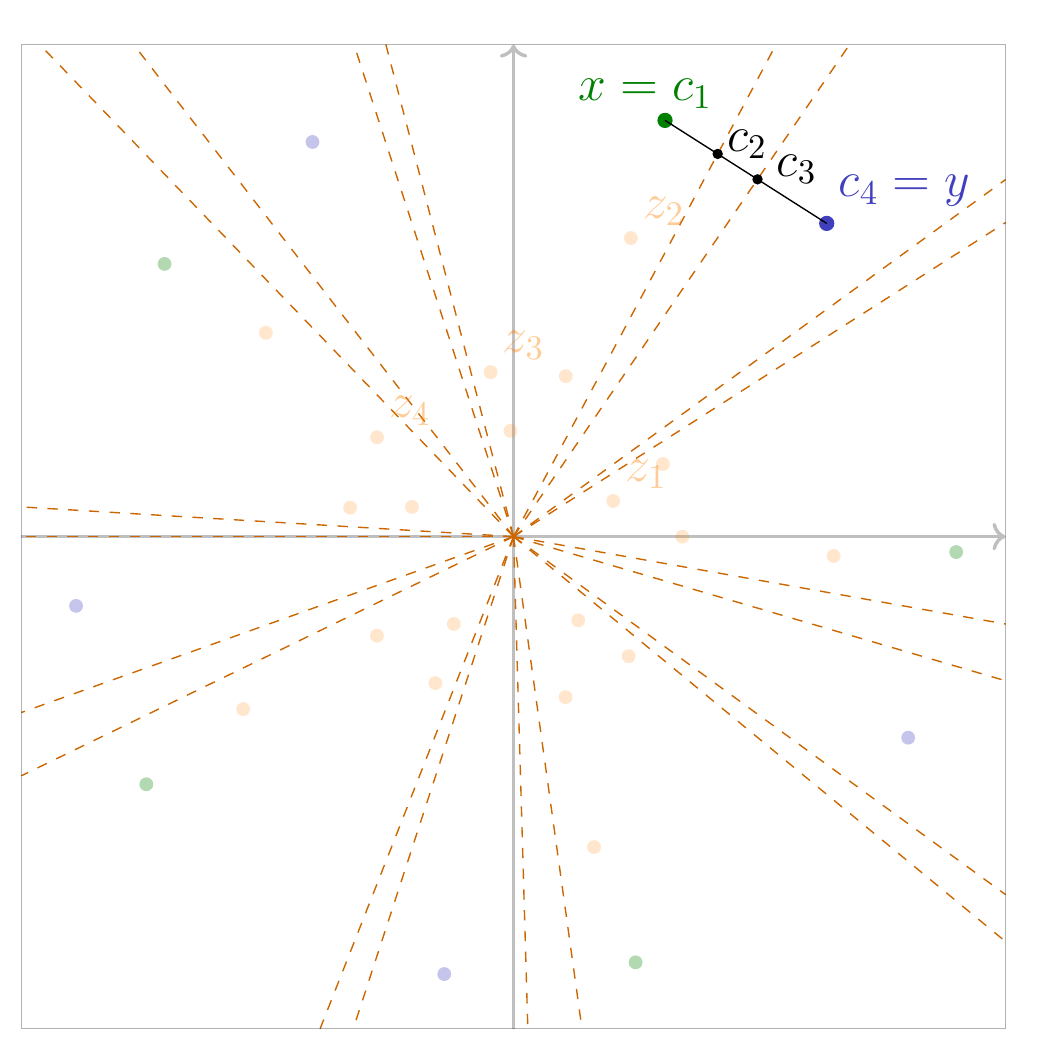}
  \caption{Illustration for the proof of Theorem~\ref{thm.upper lipschitz bound finite}. Here, $G\leq \operatorname O(2)$ consists of rotations by multiples of $\frac{2\pi}{5}$ radians. Points $x$ and $y$ are shown in green and blue colors respectively. Other points in their orbits are shown in the corresponding transparent color. Note that $d([x],[y])=\|x-y\|$. The templates $z_1,\dots, z_4$ are shown in transparent orange color along with their orbits. The union of the boundaries of their Voronoi diagrams (dashed orange lines) form the complement of $\cap_{i=1}^4 Q_{z_i}$. The line segment $[x,y]$ transverses through the connected components of $\cap_{i=1}^4 Q_{z_i}$, with $c_2$ and $c_3$ marking transition points.}
  \label{fig:upper aid}
\end{figure}
\begin{proof}
  By invariance, the index set of the left-hand quotient of~\eqref{eq.optimal upper bound finite} may be taken to be $\{(x,y)\in\mathbb R^d\times \mathbb R^d:\|x-y\|=d([x],[y])\neq 0\}$. We fix such $x,y\in\mathbb R^d$ and proceed.
  
  Let $\Pi_Q$ denote the set of connected components of $Q:=\cap_i Q_{z_i}$, and let $[x,y]$ denote the line segment joining $x$ to $y$. For $K\in \Pi_Q$, it holds that $[x,y]\cap \overline K$ is either empty or a closed line segment, being an intersection of closed convex sets, one of which is a line segment. By partitioning, it follows that there exist $K_1,\dots, K_{m-1}\in \Pi_Q$ and $c_1, \ldots, c_m \in [x,y]$ such that $x=c_1$, $y=c_m$ and $[c_j,c_{j+1}]\subseteq \overline K_j$ for each $j\in \{1,\dots,m-1\}$. See Figure~\ref{fig:upper aid} for an illustration.

%For any $j\in \{1,\dots, m-1\}$, it holds that $d([c_j],[c_{j+1}]) = \|c_j - c_{j+1}\|$. This follows from convexity of $\overline V_x$ and $\overline V_{c_{j+1}}$ and the line of argumentation is as follows: $x,y\in \overline V_x$ implies $c_{j+1}\in [x,y]\subseteq \overline V_x$ hence $x\in \overline V_{c_{j+1}}$ and finally $[c_j,c_{j+1}]\subseteq [x,c_{j+1}]\subseteq \overline V_{c_{j+1}}$ as desired.

For each $j\in \{1,\dots,m-1\}$, let $A_j := \{g_i^jz_i\}_{i=1}^n$ denote the matrix associated to $K_j\in \Pi_Q$, i.e.,\ $K_j=\cap_i V_{g_i^j z_i}$. Then for each $i\in \{1,\dots, n\}$, it holds that
\[
\begin{aligned}
  \llangle [x],[z_i]\rrangle - \llangle [y],[z_i]\rrangle = \sum_{j=1}^{m-1} \Big(\llangle [c_j],[z_i]\rrangle - \llangle [c_{j+1}],[z_i]\rrangle \Big)= \sum_{j=1}^{m-1}\Big( \langle c_j, g_i^jz_i\rangle - \langle c_{j+1},g_i^jz_i\rangle\Big).
\end{aligned}
\]
By concatenating over $i\in \{1,\dots, n\}$ and using $d([x],[y]) = \|x-y\|\neq 0$, it follows that
\[\frac{\Phi([x]) - \Phi([y])}{d([x],[y])} = \sum_{j=1}^{m-1}\frac{\|c_{j}-c_{j+1}\|}{\|x-y\|} A_j^\top\frac{c_j-c_{j+1}}{\|c_j-c_{j+1}\|}.\]
By taking norms, applying the triangle inequality, and using $\|x-y\| = \sum_{j=1}^{m-1} \|c_{j}-c_{j+1}\|$, it follows that
\[\frac{\|\Phi([x]) - \Phi([y])\|}{d([x],[y])} \leq \max_{1\leq j\leq m} \|A_j\|_{2\to 2} \leq \max_{\substack{g_1,\ldots,g_n\in G\\ \cap_i V_{g_iz_i}\neq \varnothing}}\|\{g_iz_i\}_{i=1}^n\|_{2\to2}.\]
To establish equality, take $h_1,\dots,h_n\in G$ that maximizes the right-hand side. Let $A := \{h_iz_i\}_{i=1}^n$, and select a unit vector $u$ in such a way that $\|A^\top u\|= \|A^\top\|_{2\to 2} = \|A\|_{2\to 2}$. Take any $x\in \cap_i V_{h_iz_i}\neq \varnothing$. Since $\cap_i V_{h_iz_i} \cap V_x$ is open and contains $x$, we may take $y \in \cap_i V_{h_iz_i}\cap V_x$ in such a way that $\frac{x-y}{\|x-y\|} = u$. Then $d([x],[y]) = \|x-y\|$, and we obtain
\[\frac{\|\Phi([x]) - \Phi([y])\|}{d([x],[y])} = \|A^\top u\| = \|A\|_{2\to 2}.\]
\end{proof}

As an illustration, we compute the optimal upper Lipschitz bound for a simple example.

\begin{example}
  \label{ex.upper bound of rotation}
  Take $V=\mathbb R^2$ and suppose that $G\leq \operatorname O(2)$ consists of rotations by multiples of $\frac{2\pi}{3}$ radians, generated by a counter-clockwise rotation $g$. Take $z_1 = (1,0)$ and $z_2 = (\frac{1}{2}, \frac{\sqrt 3}{2})$. Then $Q_{z_1}\cap Q_{z_2}$ has six connected components with two kinds of associated matrices: $A_j := g^j[z_1,z_2]$ and $B_j := g^j[z_1,g^{-1}z_2]$, where $j\in \{0,1,2\}$ and $[a,b]$ denotes the matrix whose columns are given by $a,b\in\mathbb R^2$. Then
  \[\|A_j\|_{2\to 2} = \left\|\begin{bmatrix}1 &  \frac{1}{2} \\ 0 & \frac{\sqrt 3}{2}\end{bmatrix}\right\|_{2\to 2} = \sqrt{\frac{3}{2}} = \left\|\begin{bmatrix}1 &  \frac{1}{2} \\ 0 & -\frac{\sqrt 3}{2}\end{bmatrix}\right\|_{2\to 2} = \|B_j\|_{2\to 2}.\]
  By Theorem~\ref{thm.upper lipschitz bound finite}, the max filter bank $\Phi\colon\mathbb{R}^d/G\to\mathbb{R}^n$ defined by $\Phi([x]):=\{\llangle [z_i],[x]\rrangle\}_{i=1}^2$ has optimal upper Lipschitz bound $\sqrt{3/2}$.
\end{example}

%Since $y\in\overline{V_x}$ is equivalent to $x\in\overline{V_y}$, the reader might expect $y\in V_x$ to be equivalent to $x\in V_y$.
%Sadly, this does not hold in general.

\subsection{Compact Groups}

In this subsection, we derive an upper Lipschitz bound for max filter banks when $G$ is compact and not necessarily finite. The bound is optimal whenever $-\operatorname{id}\in G$.

\begin{theorem}
\label{thm.upper lipschitz bound}
Given $G\leq\operatorname{O}(d)$ with closed orbits and $z_1,\ldots,z_n\in\mathbb{R}^d$, the max filter bank $\Phi\colon\mathbb{R}^d/G\to\mathbb{R}^n$ defined by $\Phi([x]):=\{\llangle [z_i],[x]\rrangle\}_{i=1}^n$ satisfies
\[
\sup_{\substack{[x],[y]\in\mathbb{R}^d/G\\{[x]\neq[y]}}}\frac{\|\Phi([x])-\Phi([y])\|}{d([x],[y])}
\leq\max_{g_1,\ldots,g_n\in G}\|\{g_iz_i\}_{i=1}^n\|_{2\to2}.
\]
%blue
\textcolor{black}{Furthermore, equality holds if $-\operatorname{id}\in G$.}
\end{theorem}

Before proceeding with the proof, we note that equality may break in cases where $-\operatorname{id}\notin G$. For example, take the setup and notation as in Example~\ref{ex.upper bound of rotation}. The optimal bound there was found to be $\sqrt{3/2}$. Since $gz_2 = -z_1$, it holds that
\[\max_{g_1,g_2\in G}\|[g_1z_1,g_2z_2]\|_{2\to2} \geq \|[z_1,gz_2]\|_{2\to 2} = \|[z_1,-z_1]\|_{2\to 2} = 2 > \sqrt{3/2}.\]
(We thank Radu Balan for bringing this example to our attention.) The looseness in the bound arises since we are maximizing over a larger set compared to Theorem~\ref{thm.upper lipschitz bound finite}.

\begin{proof}[Proof of Theorem~\ref{thm.upper lipschitz bound}]
By Proposition~\ref{prop.closed subgroups}, we may assume $G$ is closed without loss of generality.
Fix $x,y\in\mathbb{R}^d$, and define
\[
A:=\big\{i\in[n]:\llangle[z_i],[x]\rrangle\geq\llangle[z_i],[y]\rrangle\big\}.
\]
%blue
\textcolor{black}{In other words, $A$ partitions indices $i\in [n]$ based on whether $[x]$ or $[y]$ has a higher max filter with $[z_i]$.}
For each $i\in A$, we select $g_i\in G$ in such a way that $\llangle[z_i],[x]\rrangle=\langle g_iz_i,x\rangle$; such a choice exists since $G$ is compact.
Then
\begin{align*}
\llangle[z_i],[x]\rrangle-\llangle[z_i],[y]\rrangle
&=\langle g_iz_i,x\rangle-\max_{r\in G}\langle g_iz_i,ry\rangle\\
&=\min_{r\in G}\langle g_iz_i,x-ry\rangle
=\min_{r\in G}|\langle g_iz_i,x-ry\rangle|,
\end{align*}
where the last equality follows from the left-hand side being nonnegative by the definition of $A$.
For each $i\in [n]\setminus A$, we select $g_i$ in such a way that $\llangle [z_i],[y]\rrangle=\langle g_iz_i,y\rangle$.
Then
\begin{align*}
\llangle[z_i],[y]\rrangle-\llangle[z_i],[x]\rrangle
&=\langle g_iz_i,y\rangle-\max_{r\in G}\langle g_iz_i,rx\rangle\\
&=\min_{r\in G}|\langle g_iz_i,y-rx\rangle|
=\min_{r\in G}|\langle rg_iz_i,x-ry\rangle|,
\end{align*}
where the last step follows from a change of variables $r\mapsto r^{-1}$.
Define $h\colon G\to G^n$ by
\[
h(r)_i:=\left\{\begin{array}{rl}
g_i&\text{if }i\in A\\
rg_i&\text{if }i\in[n]\setminus A.
\end{array}\right.
\]
Notably, $h$ is continuous. %blue
\textcolor{black}{Roughly, we have turned a difference of max filters into an absolute value minimization over $r\in G$ of an inner product between an element of $[z_i]$, namely ${h(r)}_i z_i$,  and $x-ry$. In particular, $(\min_{r\in G}\left|\langle h(r)_iz_i,x-ry\rangle\right|)^2 = \min_{r\in G}\langle h(r)_iz_i,x-ry\rangle^2$.} We get
\begin{align}
\|\Phi([x])-\Phi([y])\|^2
\nonumber
&=\sum_{i=1}^n\min_{r\in G}\langle h(r)_iz_i,x-ry\rangle^2\\
\nonumber
&\leq\min_{r\in G}\sum_{i=1}^n\langle h(r)_iz_i,x-ry\rangle^2\\
\nonumber
&=\min_{r\in G}\|(\{h(r)_iz_i\}_{i=1}^n)^\top(x-ry)\|^2\\
\nonumber
&\leq\min_{r\in G}\|\{h(r)_iz_i\}_{i=1}^n\|_{2\to2}^2\|x-ry\|^2\\
\nonumber
&\leq\max_{g_1,\ldots,g_n\in G}\|\{g_iz_i\}_{i=1}^n\|_{2\to2}^2\cdot\min_{r\in G}\|x-ry\|^2\\
\label{eq.upper lipschitz bound}
&=\max_{g_1,\ldots,g_n\in G}\|\{g_iz_i\}_{i=1}^n\|_{2\to2}^2 \cdot d([x],[y])^2,
\end{align}
as desired. Now, suppose that $-\operatorname{id}\in G$. Then, $\llangle [z_i],[x]\rrangle^2 = \sup_{g_i\in G}\langle g_iz_i,x\rangle^2$ for all $x\in \mathbb R^d$ and $i\in [n]$. It follows that
\begin{equation}
  \label{eq.upper equality I}
  \|\Phi([x])\|^2 = \sup_{g_1,\ldots,g_n\in G}\sum_{i=1}^n\langle g_iz_i,x\rangle^2\leq \max_{g_1,\ldots,g_n\in G}\|\{g_iz_i\}_{i=1}^n\|_{2\to2}^2 \cdot \|x\|^2.
\end{equation}
Select $g_1,\ldots,g_n\in G$ in such a way that $Z:=\{g_iz_i\}_{i=1}^n$ has maximum spectral norm, and then take $x_0$ to be a top left-singular vector of $Z$.
%blue
\textcolor{black}{Then $\|\Phi([x_0])\|=\|Z\|_{2\to2} \cdot \|x_0\|$ since by \eqref{eq.upper equality I}, it holds that
  \[\|Z\|_{2\to2} \cdot \|x_0\|= \|Z^\top x_0\| \leq \|\Phi([x_0])\| \leq \|Z\|_{2\to2}\cdot \|x_0\|.\]
It follows that}
\begin{align*}
\|\Phi([x_0])-\Phi([0])\|^2 =\max_{g_1,\ldots,g_n\in G}\|\{g_iz_i\}_{i=1}^n\|_{2\to2}^2 \cdot d([x_0],[0])^2,
\end{align*}
i.e., $[x]=[x_0]$ and $[y]=[0]$ together achieve equality in the bound \eqref{eq.upper lipschitz bound}.
\end{proof}

\begin{comment}
\subsection{Lower Lipschitz bound}
\textbf{Throughout this subsection, we assume $G$ is finite.}
\end{comment}

\section{Lower Stability}
\label{sec.lower stability}
\textbf{Throughout this section, we assume $G$ is finite.} \textcolor{black}{In the first subsection, we construct a lower Lipschitz bound $\alpha$ for every max filter bank, and we prove its optimality for many examples of $G$. In the second subsection, we give a weaker bound $\tilde \alpha \leq \alpha$ that is more theoretically accessible, and we use it to estimate the distortion of max filter banks with random templates.}

\subsection{Sharp lower Lipschitz bound}
\label{subsec.sharp lower bound}
We construct $\alpha=\alpha(\{z_i\}_{i=1}^n,G)$ that serves as a lower Lipschitz bound:
  \begin{equation}
    \label{eq.alpha bound aim}
\inf_{\substack{[x],[y]\in\mathbb{R}^d/G\\{[x]\neq[y]}}}\frac{\|\Phi([x])-\Phi([y])\|}{d([x],[y])}
\geq\alpha.
  \end{equation}
We warm up with the most trivial example.
%blue
\begin{example}
  \label{ex.trivial case}
\textcolor{black}{Suppose $G=\{\operatorname{id}_V\}$. Then the max filter map is bilinear, $\llangle [x], [z_i]\rrangle = \langle x,z_i\rangle$, and the orbital distance is Euclidean, $d([x],[y])=\| x-y\|$. In this case, we may compute}
\begin{equation}
  \label{eq.trivial lower bound}
  \frac{\|\Phi([x])-\Phi([y])\|^2}{d([x],[y])^2}
=\sum_{i=1}^n\frac{\big(\langle z_i,x\rangle-\langle z_i,y\rangle\big)^2}{\|x-y\|^2}
=\sum_{i=1}^n\Big\langle z_i,\frac{x-y}{\|x-y\|}\Big\rangle^2
\geq \lambda_{\min}\Big(\sum_{i=1}^n z_iz_i^\top\Big),
\end{equation}
where the last step follows from the inequality $u^\top Au \geq \lambda_{\min}(A)$, which holds for any unit vector $u$ and symmetric matrix $A\in \mathbb R^{d\times d}$.
%blue
\textcolor{black}{As such, to obtain a positive lower Lipschitz bound, we only need to ensure that the positive semidefinite matrix $\sum_{i=1}^n z_iz_i^\top$ is invertible. This occurs if and only if $\{z_i\}_{i=1}^n$ spans $\mathbb R^d$. This in turns holds only if $n\geq d$, in which case generic $\{z_i\}_{i=1}^n$ suffice.}
\end{example}  

In the above example, the first equality in~\eqref{eq.trivial lower bound} leveraged bilinearity of the max filtering map when $G$ is trivial. While the max filtering map fails to be bilinear for nontrivial $G$, using the Voronoi cell decomposition introduced in Section~\ref{sec.upper stability} allows one to leverage linearity in the analysis. Hereinafter, we shall refer to the sets $V_x$, $\overline V_x$, and $Q_x$ characterized by \eqref{eq.open voronoi characterization}, \eqref{eq.closed voronoi characterization}, and \eqref{eq.disjoint union of cones}, respectively. In what follows, we motivate with some intuition that we make rigorous later.

First, we remark that by the continuity of the left-hand quotient of~\eqref{eq.alpha bound aim}, one may equivalently take the infimum over $x$ and $y$ in an open and dense subset $\mathcal{O}$ of $\mathbb{R}^d$.
We will find that for an appropriate choice of ``nice points'' $\mathcal{O}$, this restriction simplifies the analysis.

Take templates $z_1,\dots, z_n\in \mathbb R^d$ and $x\in \cap_i Q_{z_i}$. By the definition of $Q_{z_i}$, it holds that $x\in \cap_i V_{v_i(x)}$ for a unique choice of $v_i(x)\in [z_i]$. This implies that $v_i(x)$ is the unique element in $[z_i]$ that satisfies $\llangle [x],[z_i]\rrangle = \langle x,v_i(x)\rangle$. Stated differently, $\overline V_x \cap [z_i] = \{v_i(x)\}$. By virtue of this uniqueness, we shall include the open dense set $\cap_i Q_{z_i}$ in the definition of $\mathcal O$.

Next, we introduce linearity in our analysis by noting that $\llangle [y],[z_i]\rrangle = \langle w_i,v_i(x)\rangle$ for some $w_i \in [y]$. This defines a (choice) function $f\colon \{1,\dots,n\}\to [y]$ such that $f(i)=w_i$, and we denote by $\mathcal F(x,y)$ the space of such choice functions. In loose terms, we have thus far brought $z_i$ as close as possible to $x\in \cap_i Q_{z_i}$ and named that unique element $v_i(x) \in [z_i]$, and then we brought $y$ as close as possible to $v_i(x)$, potentially non-uniquely, and named that element $f(i)=w_i$. This affords some linearity in the max filtering map:
\[\llangle [x],[z_i]\rrangle-\llangle [y],[z_i]\rrangle = \langle v_i(x), x- f(i)\rangle.\] 
By partitioning over the preimages of $f$, we can replicate the analysis in~\eqref{eq.trivial lower bound}:
\[\|\Phi([x])-\Phi([y])\|^2
=\sum_{w\in [y]}\sum_{i\in f^{-1}(w)}\langle v_i(x),x-w\rangle^2\geq \sum_{w\in [y]} \lambda_{\min}\bigg(\sum_{i\in f^{-1}(w)}v_i(x)v_i(x)^\top\bigg)\cdot\|x-w\|^2.\]
Since $\|x-w\|^2 \geq d([x],[y])^2$, we obtain
\begin{equation}
  \label{eq.informal}
\frac{\|\Phi([x])-\Phi([y])\|}{d([x],[y])} \geq \sqrt{\sum_{w\in [y]} \lambda_{\min}\bigg(\sum_{i\in f^{-1}(w)}v_i(x)v_i(x)^\top\bigg)}.
\end{equation}
By virtue of the Weyl inequality $\lambda_{\min}(A+B)\geq \lambda_{\min}(A) + \lambda_{\min}(B)$ for symmetric matrices $A,B\in \mathbb R^{d\times d}$, we are inclined to minimize the size of $\operatorname{im}(f)\subseteq [y]$ so as to obtain more summand terms inside the $\lambda_{\min}$'s, and thus raise the lower bound on the right-hand side.

Since the choice of $f(i)=w_i$ was made in such a way that $v_i(x)\in \overline V_x \cap \overline V_{w_i}$, we are inclined to find a minimal set $S(x,y)$ such that $\overline V_x \subseteq \cup_{w \in S(x,y)}\overline V_{w}$. This way, the codomain of $f$ may be restricted to $S(x,y)$. In fact, the minimality of $S(x,y)$ entails its uniqueness, as we shall see in Lemma~\ref{lem.equivalence Sxy}. For now, we provide said unique set as a definition:

\begin{figure}[t]
  \centering
  \includegraphics[scale=0.75]{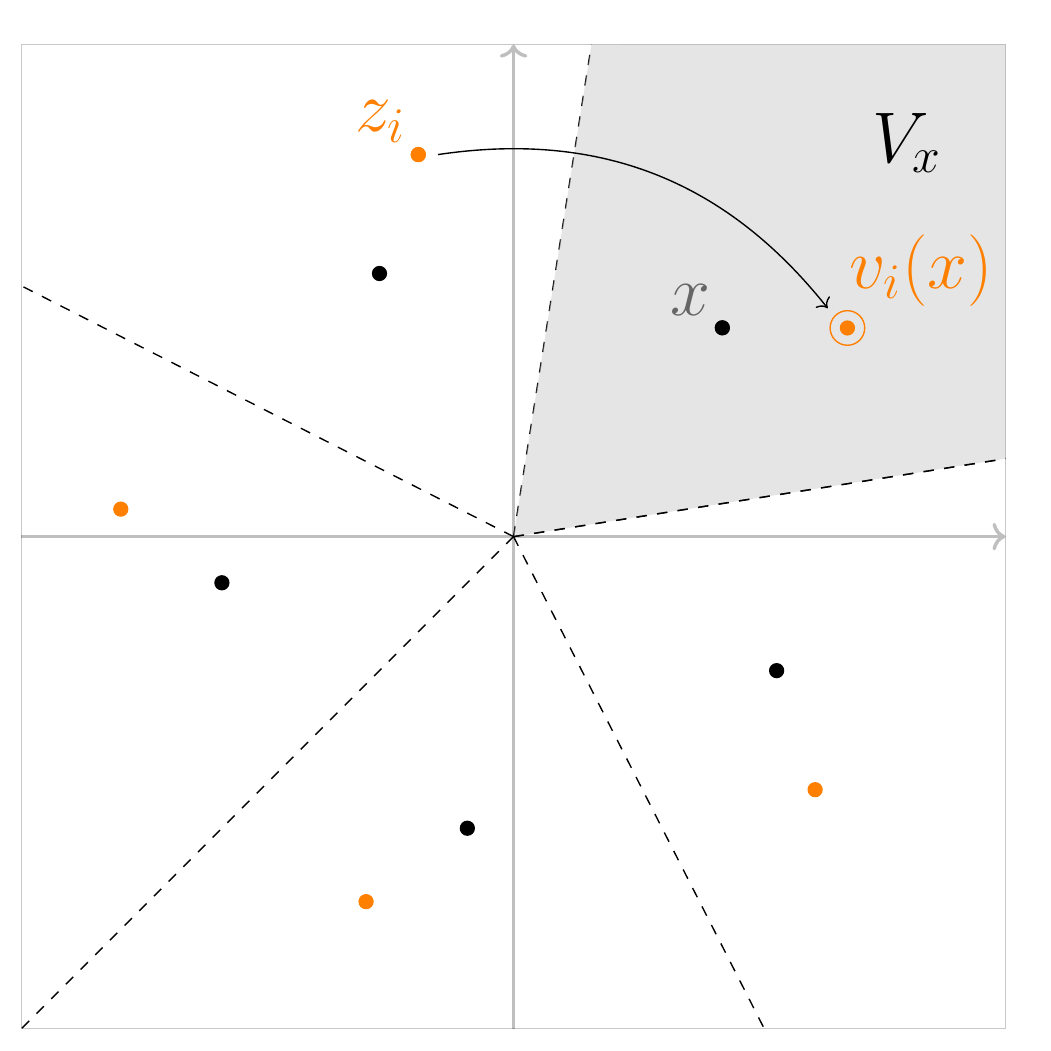}
  \includegraphics[scale=0.75]{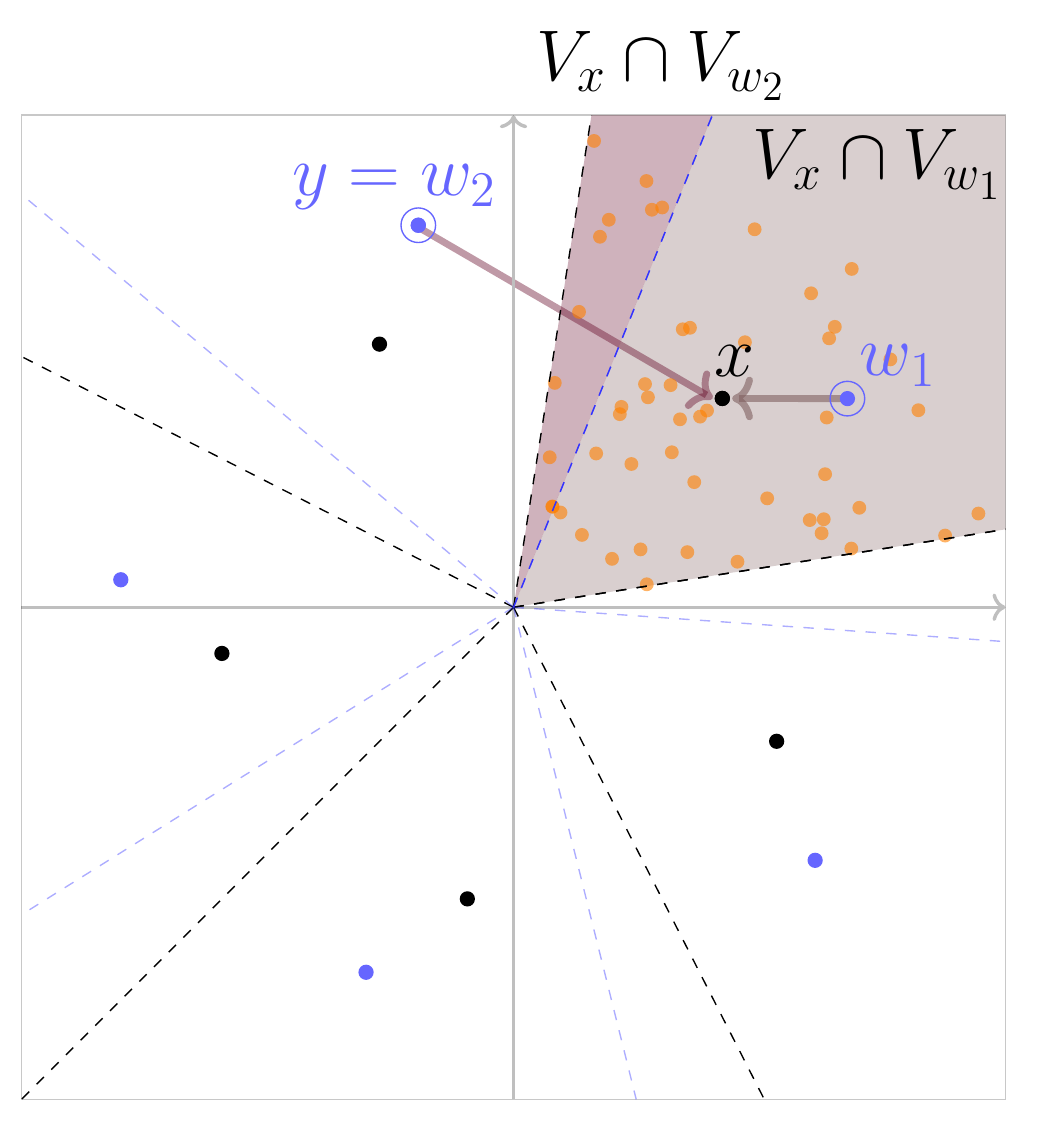}
  \caption{The group $G\leq \operatorname{O}(2)$ consists of rotations by multiples of $\frac{2\pi}{5}$ radians as in Example~\ref{ex.rotation Sxy 2d} with $m=5$. \textbf{(left)} Illustration of Corollary~\ref{cor.choice function}(a). The point $v_i(x)$ is the closest member of $z_i$'s (orange) orbit to $x$. This is the unique closest point since $x\in V_{v_i(x)}\subseteq Q_{z_i}$. \textbf{(right)} Illustration of Definition~\ref{def.S(x,y)} and Corollary~\ref{cor.choice function}. The orange dots in $V_x$ are the images $v_i(x)\in [z_i]$. Here, $S(x,y) = \{w_1,w_2\}$ are the elements of $[y]$ that have nontrivial open Voronoi overlap with $V_x$, i.e., $V_x\cap V_{w_j}\neq \varnothing$. For $f\in \mathcal F(x,y)$, one may choose $f(i) = w_j$ as long as $v_i(x)\in \overline {V_{w_j}}$.}
  \label{fig:Sxy illustration}
\end{figure}

\begin{definition}
  \label{def.S(x,y)}
  Fix a finite group $G\leq\operatorname{O}(d)$.
  For each $x,y\in\mathbb{R}^d$, we define
  \[
  S(x,y):=\{q\in[y]:V_q\cap V_x\neq\varnothing\}.
  \]
\end{definition}

  The set $S(x,y)$ is the set of elements of $[y]$ that have nontrivial open Voronoi cell overlap with $V_x$. For $x,y\in \mathbb R^d$, we claim that we have a covering $ \overline V_x\subseteq \cup_{w\in S(x,y)}\overline V_w$. In fact, we shall see in Lemma~\ref{lem.equivalence Sxy} that this covering is minimal, although we do not need this fact for now. To prove the claim, take $q\in[y]\setminus S(x,y)$. Then, $V_q\cap V_x=\varnothing$ so that $\overline V_q\cap V_x=\varnothing$ by the openness of $V_x$. As such,
  \[
V_x
\subseteq\mathbb{R}^d\setminus\bigg(\bigcup_{q\in[y]\setminus S(x,y)}\overline{V_q}\bigg)
\subseteq\bigcup_{p\in S(x,y)}\overline{V_p},
\]
and the claim follows by taking closures. With that, we may formally summarize the notational portion of our discussion thus far into the following corollary; see Figure~\ref{fig:Sxy illustration} for an illustration.

  \begin{corollary}
    \label{cor.choice function}
    Fix finite $G\leq\operatorname{O}(d)$, $z_1,\ldots,z_n\in\mathbb{R}^d$, $x\in\cap_{i=1}^n Q_{z_i}$, and $y\in\mathbb{R}^d$.
    \begin{itemize}
    \item[(a)]
    For every $i\in\{1,\ldots,n\}$, the set $\arg\max_{p\in[z_i]}\langle p,x\rangle$ consists of a single element $v_i(x)$.
    \item[(b)]
    There is a nonempty set $\mathcal{F}(x,y)$ of choice functions $f\colon\{1,\ldots,n\}\to[y]$ such that $f(i)\in S(x,y)\cap\arg\max_{q\in[y]}\langle q,v_i(x)\rangle$.
    \end{itemize}
    \end{corollary} 
  
  From the discussion following \eqref{eq.informal}, recall that we aim to minimize $|\operatorname{im}f|\leq |S(x,y)|$ for $x,y\in \cap_i Q_i$ and $f\in \mathcal F(x,y)$. We shall find that $|S(x,y)|$ gives a sharp bound on $|\operatorname{im}f|$ for $x$ and $y$ in an open and dense subset of $\mathbb R^d$. First, we warm up with computational examples of $S(x,y)$; see Figures~\ref{fig:Sxy illustration} and \ref{fig:Sxy sharpness illustration} for an illustration.

  \begin{figure}[t]
    \centering
    \includegraphics[scale=0.6]{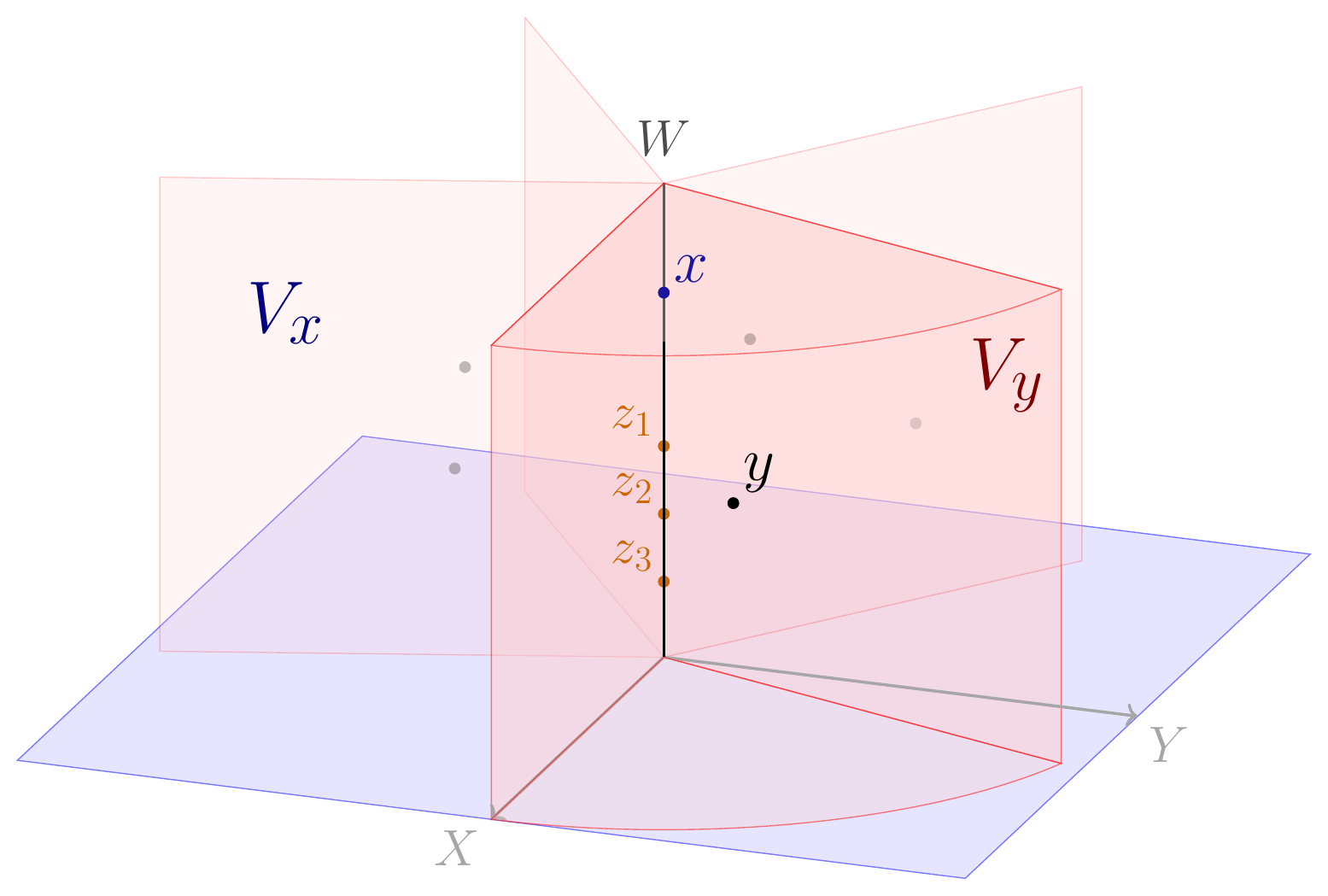}
    \includegraphics[scale=0.6]{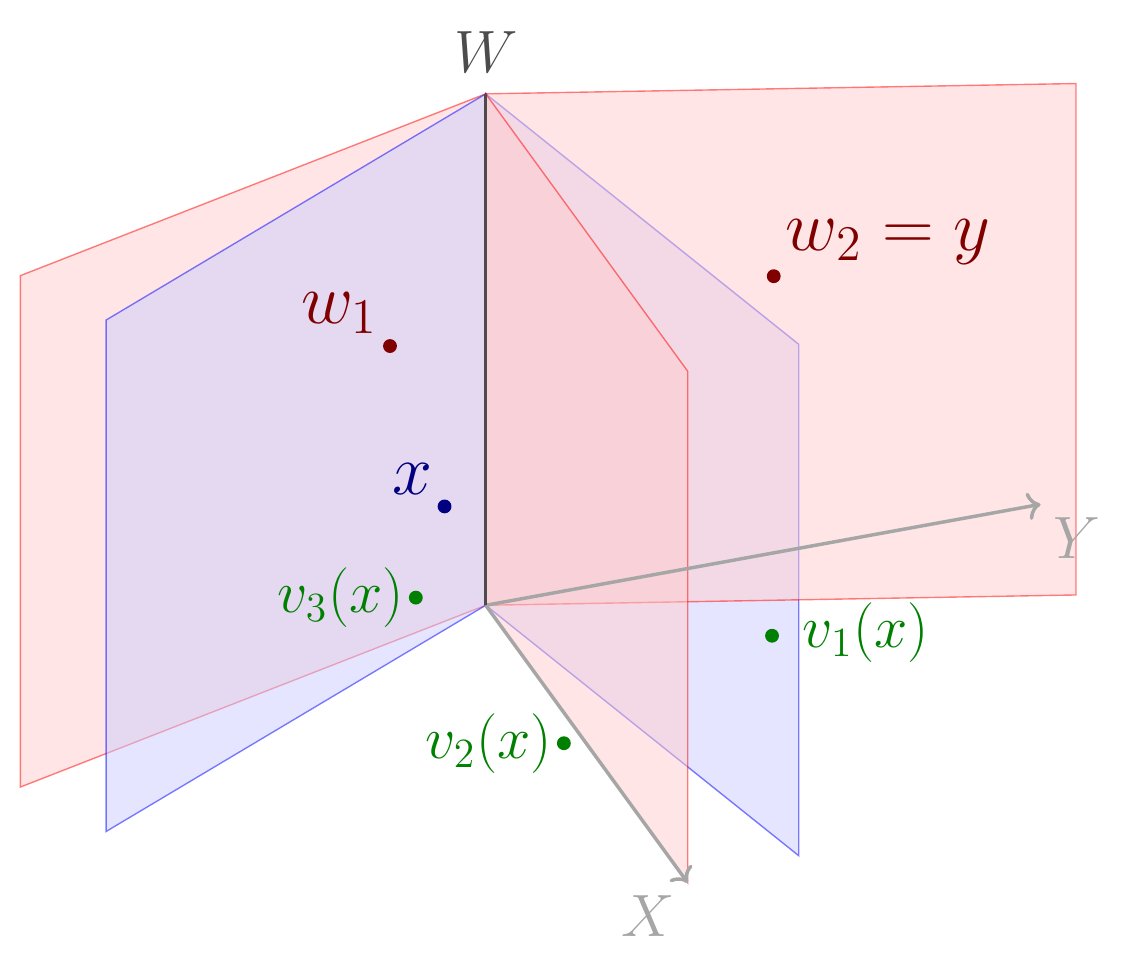}
    \caption{The group consists of rotations by multiples of $\frac{2\pi}{5}$ radians around the $z$-axis as in Example~\ref{ex.rotation Sxy 3d} with $m=5$. Both pictures serve as an illustration for the discussion that follows Example~\ref{ex.rotation Sxy 3d}. \textbf{(left)} The templates $z_1,z_2,z_3$ and the point $x$ lie on the upper half of the $z$-axis $W$ while $y\in W^c$. The Voronoi cell $V_x$ is given by the open upper half-space. While $S([x],[y]) = [y]$, there exists $f\in \mathcal F(x,y)$ such that $f(i)=y$ for each $i\in\{1,2,3\}$.
    \textbf{(right)} In this picture, $x,y\in W^c$ and $f(1) = w_2$ while $f(2)=f(3)=w_1$ so that $|\operatorname{im}(f)|=|S(x,y)|=2$, where $f$ is the singleton element of $\mathcal F(x,y)$.}
    \label{fig:Sxy sharpness illustration}
  \end{figure}

  \begin{example}
    \label{ex.rotation Sxy 2d}
    Suppose that $G\leq \operatorname{O}(2)$ consists of rotations whose angles are multiples of $\frac{2\pi}{m}$ radians, with $m\geq 2$. From Example~\ref{ex.voronoi 2d rot}, recall that $V_0=\mathbb R^2$ and that for $x\neq 0$
    \[
    V_x =
      \{p\in\mathbb R^3: \arg(p) \in (\arg(x)-\pi/m,\arg(x)+\pi/m)\bmod 2\pi\}.
    \]
    As such,
    \begin{alignat*}{3}
     & S(x,y)=\begin{cases}
          [y] & x=0\\
          \{y\} & x \in \mathbb R_{>0}\cdot y\\
          \{w_1(y),w_2(y)\} & x\notin \mathbb R_{\geq 0}\cdot y   
    \end{cases}
    & \hskip 3em &
    \text{ and }
      & \hskip 3em &
    |S(x,y)|=\begin{cases}
      m & x=0\\
      1 &  x \in \mathbb R_{>0}\cdot y\\
      2 & x\notin \mathbb R_{\geq 0}\cdot y
    \end{cases}
\end{alignat*}
  where $w_1(y)$ and $w_2(y)$ are the elements of $[y]$ that are closest and second closest to $x$, respectively.
    \end{example}
    
  \begin{example}
    \label{ex.rotation Sxy 3d}
  Let $e_1,e_2,e_3$ denote an orthonormal basis of $\mathbb R^3$ and put $W := \operatorname{span}\{e_3\}$. Suppose that $G\leq \operatorname{O}(3)$ consists of rotations with axis $W$ and angles that are multiples of $\frac{2\pi}{m}$ radians, with $m\geq 2$. Then the orthogonal projection $P$ of $\mathbb R^3$ onto $W^\perp$ intertwines the action of $G$ with the action of its restriction $G|_{W^\perp}\leq \operatorname{O}(2)$, and it holds that $V_x=P^{-1}(V_{Px})$ where $V_{Px}$ denotes the open Voronoi cell of $Px$ with respect to the group $G|_{W^\perp}$. As such, $S(x,y) = P^{-1}(S(Px,Py))\cap [y]$ where $S(Px,Py)$ is taken with respect to the group $G|_{W^\perp}$ (cf. Example~\ref{ex.rotation Sxy 2d}). Then considering the previous example, it holds that
  \[|S(x,y)|=\begin{cases}
    m & x\in W\\
    1 &  Px \in \mathbb R_{>0}\cdot Py\\
    2 & Px\notin \mathbb R_{\geq 0}\cdot Py.
  \end{cases}\]
  \end{example}

  Notably in the above example, if $z_i\in W$ for each $i$, $x\in W\subseteq \cap_i Q_{z_i} = \mathbb R^3$ and $y\notin W$, then $v_i(x) = z_i$ and any $f(i)\in S(x,y) = [y]$ satisfies $\llangle [y],[z_i]\rrangle = \langle v_i(x), f(i)\rangle$. In particular, we may take $|\operatorname{im}(f)| = 1 \ll m = |S(x,y)|$ e.g., take $f(i) = y$ for each $i$. This illustrates that the size of the covering $\overline V_x \subseteq \cup_{w\in S(x,y)}\overline V_w$, which is minimal as we shall see in Lemma~\ref{lem.equivalence Sxy}, does not inform the minimal size of $\operatorname{im}(f)$. This is an artifact of $x\in W$ being fixed by nontrivial elements in $G$.
  
  On the other hand, if we take templates $z_i\in W^c$ and $x\in \cap_i Q_{z_i}\subseteq W^c$ such that at least two $v_i(x)$ have different polar angles, then $|\operatorname{im}(f)| = |S(x,y)| = 2$ for any $y\in V_x$ such that $v_1(x)$ and $v_2(x)$ are not both in $\overline V_y$. As such, $|S(x,y)|$ gives an optimal bound for $|\operatorname{im}(f)|$ over an open set of $(x,y)\in W^c\times W^c$. See Figure~\ref{fig:Sxy sharpness illustration} for an illustration.
  
  In fact, this behavior holds for a general group $G$. The analogue of the complement of $W$ is the set of points with trivial stabilizers:
  \begin{definition}
    \label{def.principal}
    Fix a finite group $G\leq \operatorname{O}(d)$. We denote the set of \textbf{principal points} by
  \[
P(G)
:=\Big\{x\in\mathbb{R}^d:\operatorname{stab}_G(x)=\{\operatorname{id}\}\Big\}. 
\]
  \end{definition}

Note that $P(G)$ is open and dense in $\mathbb{R}^d$ since $P(G)^c$ is the (finite) union of $1$-eigenspaces of nonidentity members of $G$; in fact, this holds more generally when $G$ is closed by Theorem~1.32 in~\cite{Meinrenken:03}. The following lemma gives a Voronoi cell characterization of principal points:

\begin{lemma}
  \label{lem.characterization of P(G)}
  The following are equivalent:
  \begin{itemize}
  \item[(a)]
  $x\in P(G)$.
  \item[(b)]
  $G$ acts freely and transitively on $\{V_p\}_{p\in[x]}$.
  \item[(c)]
  $y\in V_x$ implies $x\in V_y$.
  \end{itemize}
  \end{lemma}
  
  We postpone the (technical) proof of the above lemma to Appendix~\ref{sec.proof of principal points}. With the set of principal points in hand, we are able to give equivalent characterizations of $S(x,y)$ in the following lemma, whose (technical) proof can be found in Appendix~\ref{sec.proof of Sxy}.

\begin{lemma}
  \label{lem.equivalence Sxy}
  Fix finite $G\leq\operatorname{O}(d)$, $x,y\in\mathbb{R}^d$, and $T\subseteq[y]$.
  Consider the statements:
  \begin{itemize}
  \item[(a)]
  $T\supseteq S(x,y)$.
  \item[(b)]
  $\bigcup_{p\in T}\overline{V_p}\supseteq \overline{V_x}$.
  \item[(c)]
  For every $z\in\mathbb{R}^d$, there exists $\displaystyle v\in\arg\max_{p\in[z]}\langle p,x\rangle$ such that $\displaystyle T\cap\arg\max_{q\in[y]}\langle q,v\rangle\neq\varnothing$.
  \end{itemize}
  Then (a) $\Leftrightarrow$ (b) $\Rightarrow$ (c), and furthermore, (c) $\Rightarrow$ (b) holds if $x\in P(G)$.
  \end{lemma}

  The first equivalence shows that $S(x,y)$ is the unique minimal index set that yields a covering $\bigcup_{p\in S(x,y)}\overline{V_p}\supseteq \overline{V_x}$. When $x\in P(G)$, the equivalence $(a)\Leftrightarrow (c)$ shows that in the process of bringing some $z\in\mathbb R^d$ towards $\overline V_x$ and calling that version $v\in [z]$, and then bringing $y$ towards $v$ and calling that version $w\in [y]$, we can force $w$ to be any element of $S(x,y)$ for an appropriate choice of $z$, namely $z\in G\cdot (V_x\cap V_w)$.

  This gives a generalization of the discussion following Example~\ref{ex.rotation Sxy 3d}, as we now argue. Suppose that $x,y\in P(G)\cap \bigcap_{i=1}^nQ_{z_i}$ and $S(x,y)=\{w_1,\dots,w_m\}\subseteq [y]$. Choose $z_1,\dots,z_n\in \mathbb R^d$ such that for each $j\in \{1,\dots, m\}$, there exists $i\in \{1,\dots,n\}$ with $z_i\in G\cdot (V_x\cap V_{w_j})$. Then this is an open set of choices provided $n\geq m$, and any function $f\in \mathcal F(x,y)$ satisfies $|\operatorname{im}(f)|= |S(x,y)|$.
  
  Due to this sharpness, we find that $S(x,y)$ is a ``natural'' set with which we shall proceed with our partitioned linear analysis that factors through Corollary~\ref{cor.choice function}, provided we include $P(G)$ in the set of ``nice points'' for which we now give the definition:

  \begin{definition}
    \label{def.nice points}
We define
\[\mathcal{O}(\{z_i\}_{i=1}^n,G):=P(G)\cap \bigcap_{i=1}^nQ_{z_i}.\]

\end{definition}
    
    We are now ready to enunciate our lower Lipschitz bound:
    
    \begin{theorem}
    \label{thm.lower lipschitz bound}
    Given finite $G\leq\operatorname{O}(d)$ and $z_1,\ldots,z_n\in\mathbb{R}^d$, put
    \[
    \alpha(\{z_i\}_{i=1}^n,G)
    :=\inf_{x,y\in\mathcal{O}(\{z_i\}_{i=1}^n,G)}\max_{f\in\mathcal{F}(x,y)}\bigg(\sum_{w\in S(x,y)}\lambda_{\min}\bigg(\sum_{i\in f^{-1}(w)} v_i(x)v_i(x)^\top\bigg)\bigg)^{1/2},
    \]
    where $v_i(x)$ and $\mathcal{F}(x,y)$ are defined as in Corollary~\ref{cor.choice function}.
    The max filter bank $\Phi\colon\mathbb{R}^d/G\to\mathbb{R}^n$ defined by $\Phi([x]):=\{\llangle [z_i],[x]\rrangle\}_{i=1}^n$ satisfies
    \[
    \inf_{\substack{[x],[y]\in\mathbb{R}^d/G\\{[x]\neq[y]}}}\frac{\|\Phi([x])-\Phi([y])\|}{d([x],[y])}
    \geq\alpha(\{z_i\}_{i=1}^n,G).
    \]
    \end{theorem}
    \begin{proof}
      It suffices to prove
      \[
      \frac{\|\Phi([x])-\Phi([y])\|}{d([x],[y])}
      \geq\alpha(\{z_i\}_{i=1}^n,G)
      \]
      for all $x,y\in \mathcal{O}(\{z_i\}_{i=1}^n,G)$ with $[x]\neq[y]$; this follows from the continuity of the left-hand quotient and the density of $\mathcal{O}(\{z_i\}_{i=1}^n,G)$ in $\mathbb{R}^d$.
      As such, we fix $x,y\in \mathcal{O}(\{z_i\}_{i=1}^n,G)$ with $[x]\neq[y]$, and we select $f\in\mathcal{F}(x,y)$ that maximizes
      \[
      \sum_{w\in S(x,y)}\lambda_{\min}\bigg(\sum_{i\in f^{-1}(w)} v_i(x)v_i(x)^\top\bigg).
      \]
      By Corollary~\ref{cor.choice function}, we have $\llangle [z_i],[x]\rrangle=\langle v_i(x),x\rangle$ and $\llangle [z_i],[y]\rrangle=\langle v_i(x),f(i)\rangle$, and so
      \begin{align*}
      \|\Phi([x])-\Phi([y])\|^2
      &=\sum_{i=1}^n\big(\langle v_i(x),x\rangle-\langle v_i(x),f(i)\rangle\big)^2\\
      &=\sum_{w\in S(x,y)}\sum_{i\in f^{-1}(w)}\langle v_i(x),x-w\rangle^2\\
      &=\sum_{w\in S(x,y)}\|(\{v_i(x)\}_{i\in f^{-1}(w)})^\top(x-w)\|^2\\
      &\geq\sum_{w\in S(x,y)}\lambda_{\min}\bigg(\sum_{i\in f^{-1}(w)}v_i(x)v_i(x)^\top\bigg)\cdot\|x-w\|^2\\
      &\geq \alpha(\{z_i\}_{i=1}^n,G)^2\cdot d([x],[y])^2,
      \end{align*}
      as desired.
      \end{proof}
      
      We conclude this subsection by noting that Theorem~\ref{thm.lower lipschitz bound} is often optimal, though the proof of this fact is a bit long and technical and can be found in Appendix~\ref{sec.proof of optimal lower lipschitz}.
      
      \begin{lemma}
      \label{lem.lower lipschitz optimal cases}
      Consider any $d\in\mathbb{N}$ and $G\leq\operatorname{O}(d)$ such that
      \begin{itemize}
      \item[(a)]
      $d\leq 2$,
      \item[(b)]
      $G=\{\pm\operatorname{id}\}$, or
      \item[(c)]
      $G$ is a reflection group.
      \end{itemize}
      Then for every $z_1,\ldots,z_n\in\mathbb{R}^d$, there exist $x,y\in\mathbb{R}^d$ with $[x]\neq[y]$ such that the corresponding max filter bank $\Phi\colon\mathbb{R}^d/G\to\mathbb{R}^n$ satisfies
      \[
      \|\Phi([x])-\Phi([y])\|
      =\alpha(\{z_i\}_{i=1}^n,G)\cdot d([x],[y]),
      \]
      where $\alpha(\{z_i\}_{i=1}^n,G)$ is defined in Theorem~\ref{thm.lower lipschitz bound}.
      \end{lemma}

\subsection{Second lower Lipschitz bound and distortion estimates}
\label{subsec.second lower bound and dist}

In this subsection, we use our sharp bound from the previous subsection to state a theoretically more accessible bound which we use to derive easy-to-state sufficient conditions for max filter banks to be bilipschitz.

To motivate, observe that the quantity $\alpha(\{z_i\}_{i=1}^n,G)$ in Theorem~\ref{thm.lower lipschitz bound} is positive if and only if for every $x,y\in \mathcal O(\{z_i\}_{i=1}^n,G)$, there exists $f\in \mathcal F(x,y)$ and $w\in S(x,y)$ such that $\sum_{i\in f^{-1}(w)}v_i(x)v_i(x)^\top$ is invertible, which in turn happens only if $|f^{-1}(w)| \geq d$ (cf.\ Example~\ref{ex.trivial case}). This necessary condition can be achieved via a pigeonhole principle as we shall argue now.

To loosen the bound in Theorem~\ref{thm.lower lipschitz bound}, first note that the pigeonhole principle entails that for every $x,y\in \mathcal O(\{z_i\}_{i=1}^n,G)$ and $f\in \mathcal F(x,y)$, there exists $w_P\in S(x,y)$ such that $|f^{-1}(w_P)|\geq n/|S(x,y)|$. This way, we get
\begin{equation}
  \label{eq.pigeon}
\sum_{w\in S(x,y)}\lambda_{\min}\bigg(\sum_{i\in f^{-1}(w)} v_i(x)v_i(x)^\top\bigg) \geq \lambda_{\min}\bigg(\sum_{i\in f^{-1}(w_P)} v_i(x)v_i(x)^\top\bigg).
\end{equation}

Note that $|S(x,y)|$ is the number of ``holes.'' At worst, this is bounded above by the following complexity parameter: 

\begin{definition}
  \label{def.voronoi characteristic}
The \textbf{Voronoi characteristic} of $G$ is given by
\[
\chi(G):=\max_{x,y\in P(G)}|S(x,y)|,
\]
where  $S(x,y)$ and $P(G)$ are defined in Definitions~\ref{def.S(x,y)} and~\ref{def.principal}, respectively.
\end{definition}
To get a sense for how the Voronoi characteristic scales, we consider the trivial bound $\chi(G)\leq|G|$.
While this bound is saturated by the familiar cases in which $G=\{\operatorname{id}\}$ and $G=\{\pm\operatorname{id}\}$, the bound is often very loose.
For example, we have $\chi(G)=1$ whenever $G$ is a reflection group, and $\chi(G)=2$ whenever $G$ is a nontrivial cyclic subgroup of $\operatorname{O}(2)$.
In some sense, $\chi(G)$ captures the extent to which $\mathbb{R}^d/G$ is noneuclidean.

Now at worst, the right-hand side of~\eqref{eq.pigeon}, i.e.\ the ``pigeon,'' is bounded below by the following:
\begin{definition}
  \label{def.alpha tilde}
Given $z_1,\ldots,z_n\in\mathbb{R}^d$, we define
\[
\tilde\alpha(\{z_i\}_{i=1}^n,G)
:=\min_{\substack{I\subseteq\{1,\ldots,n\}\\|I|\geq n/\chi(G)}}
\min_{\{g_i\}_{i\in I}\in G^I}\bigg(\lambda_{\min}\bigg(\sum_{i\in I}(g_iz_i)(g_iz_i)^\top\bigg)\bigg)^{1/2}.
\]
\end{definition}

Observe that $\tilde \alpha(\{z_i\}_{i=1}^n,G)$ bounds the lowest value of $\lambda_{\min}(\sum_{i\in f^{-1}(w_P)}v_i(x)v_i(x)^\top)$ obtained by choosing the worst case $v_i(x)$ and the worst case $w_P$ guaranteed by the pigeonhole principle to satisfy
\[|f^{-1}(w_P)| \geq n/|S(x,y)|\geq n/\chi(G).\]
From the above discussion, we obtain 
\begin{corollary}
  \label{cor.weaker bound}Consider any $z_1,\dots, z_n\in\mathbb R^d$ and finite group $G\leq \operatorname{O}(d)$. Then
  \[\tilde\alpha(\{z_i\}_{i=1}^n,G)\leq \alpha(\{z_i\}_{i=1}^n,G).\]
\end{corollary}
Furthermore, we have $\tilde\alpha(\{z_i\}_{i=1}^n,G)=0$ unless $n/\chi(G)>d-1$. As such, one might expect a positive lower Lipschitz bound in the regime where $n\gtrsim\chi(G)\cdot d$.
In this section, we make this explicit in two different ways.

\begin{theorem}
\label{thm.generic bilipschitz}
Fix a finite group $G\leq\operatorname{O}(d)$.
For generic $z_1,\ldots,z_n\in\mathbb{R}^d$, the max filter bank $\Phi\colon\mathbb{R}^d/G\to\mathbb{R}^n$ defined by $\Phi([x]):=\{\llangle [z_i],[x]\rrangle\}_{i=1}^n$ is bilipschitz provided $n>\chi(G)\cdot(d-1)$.
\end{theorem}

\begin{proof}
For every $\{z_i\}_{i=1}^n$, Theorem~\ref{thm.upper lipschitz bound} implies that the corresponding max filter bank is $\|\{z_i\}_{i=1}^n\|_F$-Lipschitz.
It remains to establish a lower Lipschitz bound.
First, observe that for generic $\{z_i\}_{i=1}^n$, every $d\times d$ submatrix of $\{g_iz_i\}_{i=1}^n$ is invertible for every $\{g_i\}_{i=1}^n\in G^n$.
Thus, $\alpha(\{z_i\}_{i=1}^n,G)\geq\tilde\alpha(\{z_i\}_{i=1}^n,G)>0$, and so the result follows from Theorem~\ref{thm.lower lipschitz bound}.
\end{proof}

Interestingly, since bilipschitz implies injective, Theorem~\ref{thm.generic bilipschitz} improves over the bound given in Theorem~\ref{thm.injectivity}(c) when $\chi(G)\leq 2$.
In fact, the bound $n>\chi(G)\cdot(d-1)$ is optimal when $G$ satisfies the conditions in Lemma~\ref{lem.lower lipschitz optimal cases}. 
While the above result reports how many templates suffice for the corresponding max filter bank $\Phi$ to be bilipschitz, it does not estimate the \textbf{distortion} $\kappa(\Phi)$ of $\Phi$, i.e., the quotient of the optimal upper Lipschitz bound to the optimal lower Lipschitz bound.
Our next result estimates the distortion for random templates:

\begin{theorem}
\label{thm.distortion from random templates}
Fix a finite group $G\leq\operatorname{O}(d)$ and put $m:=|G|$ and $\chi:=\chi(G)$.
Select $\lambda_0$ and $n\in\mathbb{N}$ such that
\[
\lambda
:=\frac{n}{\chi d}
\geq\lambda_0
>1,
\]
and draw templates $z_1,\ldots,z_n\in\mathbb{R}^d$ with independent standard Gaussian entries.
There exists an absolute constant $C>0$ and $c=c(\lambda_0)>0$ such that the max filter bank $\Phi\colon\mathbb{R}^d/G\to\mathbb{R}^n$ defined by $\Phi([x]):=\{\llangle [z_i],[x]\rrangle\}_{i=1}^n$ has distortion
\[
\kappa(\Phi)
\leq \Big(C\chi^{3/2}m\log^{1/2}(em)\Big)^{1+c\lambda^{-1/2}}
\]
with probability at least $1-3e^{-d\sqrt{\lambda}}$. %blue
\textcolor{black}{In fact, we may take $C=4e^{3/2}$ and $c(\lambda_0) := 2 + \frac{\sqrt \lambda_0 +2}{\lambda_0-1}$.}
\end{theorem}

In Theorem~\ref{thm.distortion from random templates}, observe that one may take $c(\cdot)$ to be a decreasing function such that $\lim_{\lambda_0\to1}c(\lambda_0)=\infty$ and $\lim_{\lambda_0\to\infty}c(\lambda_0)=2>0$.
Considering $\kappa(\Phi)\geq1$, Theorem~\ref{thm.distortion from random templates} is optimal up to constant factors for groups of bounded order provided $\lambda$ is bounded away from $1$.
Theorem~18 in~\cite{CahillIMP:22} estimates the distortion of the max filter bank corresponding to $n$ templates drawn uniformly from the sphere in $\mathbb{R}^d$.
Their bound on the distortion degrades as $n$ gets large, and optimizing $n$ in their analysis gives the estimate
\[
\kappa(\Phi)
\leq C_0m^3d^{1/2}(d\log d+d\log m+\log^2m)^{1/2}.
\]
Notably, this optimized bound depends on the dimension $d$ of the space.
For comparison, optimizing $n$ in Theorem~\ref{thm.distortion from random templates} gives the improvement
\[
\kappa(\Phi)
\leq C_1\chi^{3/2+\epsilon}m^{1+\epsilon}
\leq C_1m^{5/2+\epsilon'}.
\]
In fact, when $G$ is a %blue
\textcolor{black}{dihedral} subgroup of $\operatorname{O}(2)$, we have $\chi=1$, and so Theorem~7(d) in~\cite{MixonP:22} implies that this optimized bound is optimal up to $O(m^\epsilon)$ factors.

\begin{proof}[Proof of Theorem~\ref{thm.distortion from random templates}]
%blue
\textcolor{black}{We give the main steps in the proof and postpone details to Appendix~\ref{sec.proof of distortion from random templates}.} Consider the random variables
\[
A:=\tilde\alpha(\{z_i\}_{i=1}^n,G),
\qquad
B:=\max_{g_1,\ldots,g_n\in G}\|\{g_iz_i\}_{i=1}^n\|_{2\to2}.
\]
We bound $B$ using Gaussian concentration (see Corollary~5.35 in~\cite{Vershynin:12}, for example):
\[
\mathbb{P}\Big\{B>\sqrt{d}+\sqrt{n}+s\Big\}
\leq m^n\cdot \mathbb{P}\Big\{\|\{z_i\}_{i=1}^n\|_{2\to2}>\sqrt{d}+\sqrt{n}+s\Big\}
\leq m^n\cdot e^{-s^2/2},
\]
for any $s\geq0$.
Taking $s:=\sqrt{2(n\log m+d\sqrt{\lambda})}$ then implies
\begin{equation}
\label{eq.random upper bound}
\mathbb{P}\Big\{B>4\sqrt{n\log (em)}\Big\}
\leq \mathbb{P}\Big\{B>\sqrt{d}+\sqrt{n}+s\Big\}
\leq m^n\cdot e^{-s^2/2}
= e^{-d\sqrt{\lambda}}.
\end{equation}
Next, we bound $A$.
To this end, (the proof of) Theorem~1.3 in~\cite{Wang:18} gives that for a $k\times \ell$ standard Gaussian matrix $X$ of aspect ratio $\lambda=\frac{\ell}{k}>1$, then for any $t\geq1$, it holds that
\begin{equation}
\label{eq.random lower sigma bound}
\sigma_{\min}(X)
\geq \sigma(\ell,\lambda,t)
:=\tfrac{1}{\sqrt{e}}(1-\tfrac{1}{\lambda})(\tfrac{1}{2(2+\sqrt{2})\sqrt{e}\cdot \lambda})^{1/(\lambda-1)}\cdot (\tfrac{1}{\sqrt{t}})^{1/(\lambda-1)}e^{-t\lambda/(\lambda-1)}\cdot \sqrt{\ell}
\end{equation}
with probability at least $1-2e^{-\ell t}$.
(Here, we take $\mu$ in~\cite{Wang:18} to be $t\lambda$.)
As such, combining with the union bound gives
\begin{align}
\mathbb{P}\Big\{A<\sigma(\lceil\tfrac{n}{\chi}\rceil,\lambda,t)\Big\}
&\leq\tbinom{n}{\lceil n/\chi\rceil}\cdot m^{\lceil n/\chi\rceil}\cdot\mathbb{P}\Big\{\sigma_{\min}(\{z_i\}_{i=1}^{\lceil n/\chi\rceil})<\sigma(\lceil\tfrac{n}{\chi}\rceil,\lambda,t)\Big\}\nonumber\\
&\leq 2\operatorname{exp}\Big((\tfrac{n}{\chi}+1)\log(e\chi m)-\tfrac{nt}{\chi}\Big)\label{eq.random lower proba bound I}\\
&\leq 2e^{-d\sqrt{\lambda}},\label{eq.random lower proba bound II}
\end{align}
where the last step follows from taking $t:=(1+\frac{2}{\sqrt{\lambda}})\log(e\chi m)$.
Overall, we have
\begin{equation}
\label{eq.random distortion bound}
\kappa(\Phi)
\leq \frac{B}{A}
\leq \frac{4\sqrt{n\log (em)}}{\sigma(\lceil\tfrac{n}{\chi}\rceil,\lambda,(1+\frac{2}{\sqrt{\lambda}})\log(e\chi m))}
\end{equation}
with probability at least $1-3e^{-d\sqrt{\lambda}}$, and the result follows.
\end{proof}

\section{Positive definiteness}
\label{sec.pos def}

Recall that a \textbf{positive definite kernel} on a nonempty set $\mathcal{X}$ is a symmetric function $K\colon \mathcal{X}\times\mathcal{X}\to\mathbb{R}$ such that for every $n\in\mathbb{N}$, $x_1,\ldots,x_n\in\mathcal{X}$, and $c_1,\ldots,c_n\in\mathbb{R}$, it holds that
\[
\sum_{i=1}^n\sum_{j=1}^n c_ic_jK(x_i,x_j)
\geq0.
\]
By the Moore--Aronszajn theorem, $K$ is a positive definite kernel if and only if there exists a feature map $f\colon\mathcal{X}\to\mathcal{H}$ to some Hilbert space $\mathcal{H}$ such that
\[
K(x,y)
=\langle f(x),f(y)\rangle_\mathcal{H}
\]
for all $x,y\in\mathcal{X}$.
Furthermore, one may take $f\colon x\mapsto K(\cdot,x)$, in which case $\mathcal{H}$ is the reproducing kernel Hilbert space given by the completion of $\operatorname{span}\{K(\cdot,x)\}_{x\in\mathcal{X}}$ with respect to the inner product defined by $\langle K(\cdot,x),K(\cdot,y)\rangle:=K(x,y)$ and extended bilinearly.
Positive definite kernels are useful in machine learning since a classifier defined on the feature space $\mathcal{H}$ can be neatly expressed in the original space $\mathcal{X}$ by virtue of the so-called \textit{kernel trick}.

We are inclined to consider $\mathcal{X}:=V/G$ and $K([x],[y]):=\llangle [x],[y]\rrangle$, where $V$ is a finite-dimensional real inner product space and $G$ is a subgroup of $\operatorname{O}(V)$ with closed orbits.
In particular, is max filtering positive definite? In this section, we shall characterize the groups $G$ for which this is the case. We warm up with examples.

\begin{comment}
Considering $\llangle [x],[y]\rrangle\geq\langle x,y\rangle$, one might be persuaded by the following argument:
\[
\sum_{i=1}^n\sum_{j=1}^n c_ic_j\llangle [x_i],[x_j]\rrangle
\geq \sum_{i=1}^n\sum_{j=1}^n c_ic_j\langle x_i,x_j\rangle
=\bigg\|\sum_{i=1}^n c_ix_i\bigg\|^2
\geq0.
\]
Unfortunately, the above ``proof'' contains an error that the reader is invited to identify.
Regardless, there are examples in which max filtering is legitimately positive definite:
\end{comment}

\begin{example}
Take $V:=\mathbb{R}^d$, let $G\leq\operatorname{O}(d)$ be a finite reflection group, let $C$ denote any Weyl chamber of $G$, and define $f$ to map $[x]\in\mathbb{R}^d/G$ to the unique member of $[x]\cap\overline{C}$.
Then $\llangle [x],[y]\rrangle=\langle f([x]),f([y])\rangle$ since $C = V_x= V_y$.
For instance, if $G$ is the image of the standard representation of the symmetric group $S_d$, then one may take $f([x]):=\operatorname{sort}(x)$.
\end{example}

\begin{example}
\label{ex.symmetric matrices}
Take $V$ to be the Hilbert space of $d\times d$ symmetric matrices with inner product $\langle X,Y\rangle:=\operatorname{tr}(XY)$, and let $G$ be the image of the representation $\rho\colon\operatorname{O}(d)\to\operatorname{O}(V)$ defined by $\rho(Q)(X):=QXQ^{-1}$.
Then the symmetric von Neumann trace inequality (Proposition~16 in~\cite{MixonP:22}) implies that $\llangle [X],[Y]\rrangle=\langle f([X]),f([Y])\rangle$, where $f([X]):=\operatorname{sort}(\operatorname{eig}(X))$.
\end{example}

The examples above are related in the sense that the conjugation orbit of $X$ corresponds to the permutation orbit of $\operatorname{eig}(X)$. In fact, we shall find that max filtering is positive definite if and only if $G$ is polar with a generalized Weyl group generated by reflections.

\begin{definition}
  Consider $G\leq\operatorname{O}(d)$ with closed orbits.
  We say $G$ is \textbf{polar} if there exists a subspace $\Sigma\leq\mathbb{R}^d$, called a \textbf{section}, such that $\Sigma$ orthogonally intersects every orbit of $G$, i.e., $G\cdot\Sigma=\mathbb{R}^d$ and for every $x\in\Sigma$, the tangent space of the orbit $[x]$ at $x$ resides in the orthogonal complement $\Sigma^\perp$.
\end{definition}
  
The following result characterizes polar groups as those whose orbit spaces arise as quotients of finite groups. Part~(a) is given by Proposition~9.2 and Corollary~9.7 in~\cite{Michor:97}, while part~(b) is given by Section~2.3 in~\cite{GordoskiL:14}.

\begin{proposition}\
\label{prop.polar}
\begin{itemize}
\item[(a)]
If $G\leq\operatorname{O}(d)$ is polar with section $\Sigma$, then $\mathbb{R}^d/G$ is isometric to $\Sigma/W(\Sigma,G)$, where $W(\Sigma,G)$ is the largest subgroup of $G$ under which $\Sigma$ is invariant.
Furthermore, $W(\Sigma,G)$ is finite.
\item[(b)]
If $G\leq\operatorname{O}(d)$ has the property that $\mathbb{R}^d/G$ is isometric to $\mathbb{R}^k/H$ for some $k\in\mathbb{N}$ and some finite $H\leq\operatorname{O}(k)$, then $G$ is polar.
\end{itemize}
\end{proposition}

We call $W(\Sigma,G)$ the \textbf{generalized Weyl group} of $G$ associated to $\Sigma$.

  \begin{example}
  Consider $V$ and $G$ as in Example~\ref{ex.symmetric matrices}.
  We claim that $G$ is polar, and that the corresponding section $\Sigma\leq V$ is the subspace of diagonal matrices.
  First, $\Sigma$ intersects every orbit of $G$ by the spectral theorem.
  To show that this intersection is orthogonal, we need to identify the tangent space of the orbit $[X]$ at some $X\in\Sigma$.
  To this end, consider any curve $t\mapsto Q(t)XQ(t)^\top$ with $Q(0)=I_d$ and $Q'(0)\in\mathfrak{so}(d)$, i.e., $A:=Q'(0)$ is antisymmetric.
  Then the tangent vector to our curve at $t=0$ is given by
  \[
  Q'(0)XQ(0)^\top+Q(0)XQ'(0)^\top
  =AX+XA^\top.
  \]
  Every member of the tangent space of $[X]$ at $X$ takes this form; see Proposition~1.20 in~\cite{Meinrenken:03}.
  For every $D\in\Sigma$, we have $XD=DX$, and so $\operatorname{tr}(DAX)=\operatorname{tr}(XDA)=\operatorname{tr}(DXA)$.
  Thus,
  \[
  \langle D,AX+XA^\top\rangle
  =\operatorname{tr}(DAX)+\operatorname{tr}(DXA^\top)
  =\operatorname{tr}(DX(A+A^\top))
  =0.
  \]
  This verifies our claim that $G$ is polar with section $\Sigma$. The generalized Weyl group is given by $W(\Sigma,G)= S_d$ and $V/G$ is isometric to $\Sigma/S_d$. As we will see in Theorem~\ref{thm.max filtering pos def}, the fact that $S_d$ is a reflection group corresponds to the observation in Example~\ref{ex.symmetric matrices} that max filtering is positive definite in this case.
  \end{example}
  
  We are now ready to state the main result of this section: 
  
  \begin{theorem}
  \label{thm.max filtering pos def}
  For $G\leq\operatorname{O}(d)$ with closed orbits, the following are equivalent:
  \begin{itemize}
  \item[(a)]
  Max filtering is positive definite.
  \item[(b)]
  $G$ is polar with section $\Sigma$ such that $\mathbb{R}^d/G$ is isometric to $\Sigma/W(\Sigma,G)$. Furthermore, $W(\Sigma,G)$ is a finite reflection group.
  \end{itemize}
  \end{theorem}

The rest of this section is dedicated to proving the above result. To prove the implication (a)$\implies$(b), we will show that positive definiteness of max filtering has two intermediate implications: (1) $G$ is polar with section $\Sigma$ and (2) minimal geodesics in $\mathbb R^d/G\cong \Sigma/W(\Sigma, G)$ are unique. By virtue of that, we begin by defining the notion of minimal geodesics in a metric space.

\begin{definition}
Given a metric space $(M,d)$ and $L>0$, we say a curve $\gamma\colon [0,L]\to M$ is a \textbf{minimal geodesic} if $d(\gamma(s),\gamma(t))=|s-t|$ for all $s,t\in [0,L]$, and given $x,y\in M$ we let $C(x,y)$ denote the set of all minimal geodesics from $x$ to $y$.
\end{definition}
When $M=\mathbb R^d/G$, we have a characterization: 
\begin{lemma}
\label{lem.minimal geodesic bijection}
Take $G\leq\operatorname{O}(d)$ with closed orbits.
For each $x,y\in\mathbb{R}^d$, there exists a bijection
\[
C([x],[y])
\longrightarrow\frac{\arg\min_{q\in[y]}\|q-x\|}{\operatorname{stab}_G(x)}.
\]
\end{lemma}

\begin{proof}
Proposition~3.1(2) in~\cite{AlekseevskyKLM:03} gives that for every $\gamma\in C([x],[y])$, there exists a unit-speed parameterization $\tilde\gamma$ of the straight line segment from $x$ to some $p\in\arg\min_{q\in[y]}\|q-x\|$ such that $\gamma(t)=[\tilde\gamma(t)]$ for all $t$.
Consider the minimal geodesics in $\mathbb{R}^d$ that traverse from a member of $[x]$ to a member of $[y]$, and whose projection onto the orbit space $\mathbb{R}^d/G$ equals $\gamma$.
By Proposition~3.1(4) in~\cite{AlekseevskyKLM:03}, these are precisely the minimal geodesics obtained by applying a member of $G$ to $\tilde\gamma$.
The subset of minimal geodesics that originate at $x\in[x]$ are obtained by applying $\operatorname{stab}_G(x)$ to $\tilde\gamma$.
The result follows.
\end{proof}

The following lemma shows that the orbit space of a finite group $G$ has unique minimal geodesics between every pair of its points if and only if $G$ is a finite reflection group.

\begin{lemma}
\label{lem.reflection group characterization}
For finite $G\leq\operatorname{O}(d)$, the following are equivalent:
\begin{itemize}
\item[(a)]
$|C([x],[y])|=1$ for all $x,y\in\mathbb{R}^d$.
\item[(b)]
For every $x\in P(G)$, it holds that $Q_x=P(G)$.
\item[(c)]
$G$ is a reflection group.
\item[(d)]
$\chi(G)=1$.
\end{itemize}
\end{lemma}

\begin{proof}
(a)$\Rightarrow$(b).
Fix $x\in P(G)$.
Then by Lemma~\ref{lem.characterization of P(G)}(b), the action of $G$ on $Q_x$ is free, and so $Q_x\subseteq P(G)$.
For the other subset inclusion, select $y\in P(G)$.
Then $|C([x],[y])|=1$, and since $x\in P(G)$, we have $|\operatorname{stab}_G(x)|=1$, and so Lemma~\ref{lem.minimal geodesic bijection} implies
\[
|\arg\max_{q\in[y]}\langle q,x\rangle|
=|\arg\min_{q\in[y]}\|q-x\||
=1.
\]
Thus, $y\in Q_x$ by \eqref{eq.disjoint union of cones}.

(b)$\Rightarrow$(c).
By assumption and since $Q_x$ is the disjoint union of open Voronoi polyhedral cones, $P(G)^c$ is contained in a union of hyperplanes $H_i$ such that for every $x\in P(G)$ and every $i$, there exist $p,q\in[x]$ with $p\neq q$ such that $H_i$ is the orthogonal complement of $p-q$.
Fix $i$, and let $r_i$ denote the reflection that fixes $H_i$.
Then for every $x\in P(G)$, there exist $a_x,b_x\in[x]$ with $a_x\neq b_x$ such that $r_ia_x=b_x$.
We claim that $r_i\in G$.

Put $n:=|G|(d-1)+1$.
Select generic $x_1,\ldots,x_n\in P(G)$ so that every $d\times d$ submatrix of $\{g_jx_j\}_{i=1}^n$ is invertible for every $\{g_j\}_{j=1}^n\in G^n$.
Then every $d\times d$ submatrix of $\{a_{x_j}\}_{j=1}^n$ is invertible.
By pigeonhole, there exist $g\in G$ and $J\subseteq\{1,\ldots,n\}$ such that $ga_{x_j}=b_{x_j}=r_ia_{x_j}$ for all $j\in J$ and $\{a_{x_j}\}_{j\in J}$ is a basis for $\mathbb{R}^d$.
Then $g=r_i$, and so $r_i\in G$.

Thus, $G$ contains the subgroup $R$ generated by the reflections $r_i$.
To show $R=G$, first note that $R\leq G$ implies $P(G)\subseteq P(R)$.
Furthermore, by construction, $P(G)^c\subseteq \bigcup_iH_i\subseteq P(R)^c$.
It follows that $P(R)=P(G)$.
As a reflection group, $R$ acts freely and transitively on the connected components of $P(R)$, namely, the Weyl chambers of $R$.
Also, $G$ acts freely and transitively on the connected components of $Q_x=P(G)$ by Lemma~\ref{lem.characterization of P(G)}(b).
Thus, the sizes of $R$ and $G$ both equal the number of connected components, and so $R=G$.

(c)$\Rightarrow$(a).
Suppose $G$ is a reflection group, and let $C\subseteq\mathbb{R}^d$ denote any Weyl chamber of $G$.
By Lemma~8 in~\cite{MixonP:22}, $\mathbb{R}^d/G$ is isometrically isomorphic to the convex set $\overline{C}$, and so there is a unique minimal geodesic between any $[x],[y]\in\mathbb{R}^d/G$.

(b)$\Rightarrow$(d).
Select $x,y\in P(G)$.
Then by assumption, $Q_x=P(G)=Q_y$, meaning for every $p\in[x]$ and $q\in[y]$, either $V_p\cap V_q=\varnothing$ or $V_p=V_q$.
By Lemma~\ref{lem.characterization of P(G)}(b), $G$ acts freely and transitively on $\{V_q\}_{q\in[y]}$, and so there is a unique $q\in[y]$ such that $V_q\cap V_x\neq\varnothing$.
Thus, $|S(x,y)|=1$.
Since $x$ and $y$ were arbitrary, the result follows.

(d)$\Rightarrow$(b).
Select $x,y\in P(G)$.
By assumption $|S(x,y)|\leq 1$.
Since $V_x$ is open and $Q_y$ is both open and dense, it follows that $|S(x,y)|\geq1$.
Thus, there is a unique $q\in[y]$ such that $V_q$ intersects nontrivially with $V_x$.
It follows that $V_x$ is contained in $(\bigcup_{p\in[y]\setminus\{q\}}V_p)^c$, the interior of which is $V_q$.
As such, $V_x\subseteq V_q$.
Taking the orbit under $G$ gives $Q_x\subseteq Q_y$.
By symmetry, $Q_y\subseteq Q_x$.
Since $x$ and $y$ were arbitrary, there exists $S$ such that $Q_x=S$ for all $x\in P(G)$.
Since $x\in Q_x=S$ for all $x\in P(G)$, we have $P(G)\subseteq S$.
By Lemma~\ref{lem.characterization of P(G)}(b), the action of $G$ on $Q_x$ is free whenever $x\in P(G)$, and so $S=Q_x\subseteq P(G)$.
The result follows.
\end{proof}

%blue
To prove Theorem~\ref{thm.max filtering pos def}, we will additionally make use of the following two propositions from the literature. For the following proposition, see Lemma~2.1 in~\cite{BergCR:84} combined with Theorem~1 in~\cite{HeinLB:04}.

\begin{proposition}
\label{prop.pos def embedding}
For $G\leq\operatorname{O}(d)$ with closed orbits, the following are equivalent:
\begin{itemize}
\item[(a)]
Max filtering is positive definite.
\item[(b)]
There exists a Hilbert space $\mathcal{H}$ and a continuous isometric embedding $\Psi\colon \mathbb{R}^d/G\to\mathcal{H}$.
\end{itemize}
\end{proposition}

The next proposition gives information about the projection of the principal points onto the orbit space.
Here, (a) follows from Theorem~3.82 in~\cite{AlexandrinoB:15}, (b) follows from Exercise~3.81 in~\cite{AlexandrinoB:15}, and (c) follows from Section~2.3 in~\cite{GordoskiL:14}.

\begin{proposition}
\label{prop.[P(G)]}
Suppose $G\leq\operatorname{O}(d)$ has closed orbits.
Denote $[P(G)]:=\{[x]:x\in P(G)\}$.
\begin{itemize}
\item[(a)]
$[P(G)]$ is an open and dense manifold in $\mathbb{R}^d/G$.
\item[(b)]
$[P(G)]$ admits a Riemannian structure whose geodesic metric is the quotient metric.
\item[(c)]
$[P(G)]$ has constant curvature $0$ if and only if $G$ is polar.
\end{itemize}
\end{proposition}

We are now ready to prove the main result.

\begin{proof}[Proof of Theorem~\ref{thm.max filtering pos def}]
(a)$\Rightarrow$(b).
By Proposition~\ref{prop.pos def embedding}, there exists a continuous isometric embedding $\Psi$ of $\mathbb{R}^d/G$ into a Hilbert space $\mathcal{H}$.
By Proposition~\ref{prop.[P(G)]}, $[P(G)]$ is an open and dense Riemannian manifold in $\mathbb{R}^d/G$.
Denote $n:=\operatorname{dim}([P(G)])$ and fix $x\in [P(G)]$.
We assume without loss of generality that $\Psi(x)=0\in\mathcal{H}$ by translating if necessary.
By Lemma~5.12 and Proposition~6.1 in~\cite{Lee:06}, there is a \textit{uniformly normal neighborhood} $U$ of $x$ in $[P(G)]$, which in turn implies that $U$ is an open neighborhood of $x$ with homeomorphism $\phi\colon U\to\mathbb{R}^n$, and furthermore, $U$ is \textit{geodesically convex} in the sense that for every $p,q\in U$, there exists a minimal geodesic in $U$ that connects $p$ to $q$.

Since $\Psi$ is an isometry, $\Psi(U)$ is convex.
Select an open ball $B$ in the relative interior of $\Psi(U)$.
Since $\phi\circ\Psi^{-1}$ continuously embeds $B$ into $\mathbb{R}^n$, invariance of domain implies that the affine dimension of $B$ is at most $n$.
Since $0\in\Psi(U)$, it follows that the dimension of the span of $\Psi(U)$ is at most $n$.
This shows that $U$ is isometric to a subset of a finite-dimensional Euclidean space, meaning $U$ has constant curvature $0$.
Proposition~\ref{prop.[P(G)]}(c) then gives that $G$ is polar.
Note that $|C([x],[y])|=1$ for every $x,y\in\mathbb{R}^d$ due to the isometric embedding $\Psi$.
The result then follows from Proposition~\ref{prop.polar}(a) and Lemma~\ref{lem.reflection group characterization}.

(b)$\Rightarrow$(a).
By assumption, $\mathbb{R}^d/G$ is isometric to $\Sigma/H$, which in turn is isometric to a closed Weyl chamber of $H$.
This gives a continuous isometric embedding $\Psi\colon\mathbb{R}^d/G\to\Sigma$, and so the result follows from Proposition~\ref{prop.pos def embedding}.
\end{proof}

\section{Discussion}
\label{sec.discussion}

In this paper, we studied the injectivity, stability, and positive definiteness of max filtering.
We established in Theorem~\ref{thm.injectivity} that $2d$ generic templates suffice for the corresponding max filter bank to be injective.
Later, we established in Theorem~\ref{thm.generic bilipschitz} that $\chi(G)\cdot(d-1)+1$ generic templates suffice for the corresponding max filter bank to be bilipschitz (and therefore injective).
The latter result requires fewer templates than the former when $\chi(G)\leq2$, and this gives the correct threshold for several choices of $G$.
It would be interesting to determine this \textit{generic injectivity threshold} for every $G$.
In the case where $G\leq\operatorname{U}(k)\leq\operatorname{O}(2k)$ is defined by $G:=\{\omega I_k:\omega\in\mathbb{C},|\omega|=1\}$, this threshold $\tau(k)$ is an open problem in (complex) phase retrieval~\cite{BandeiraCMN:14}.
It is known that $\tau(k)\leq 4k-4$ with equality when $k$ is one more than a power of $2$~\cite{ConcaEHV:15}, and otherwise $\tau(k)>4k-4-2\alpha$, where $\alpha$ is the number of $1$'s in the binary expansion of $k-1$~\cite{HeinosaariMW:13}.
In the case where $k=4$, it appears as though a nongeneric choice of $4k-5=11$ templates gives an injective max filter bank~\cite{Vinzant:15}.

Next, Theorem~\ref{thm.lower lipschitz bound} gives a lower Lipschitz bound for max filter banks when $G$ is finite.
Furthermore, by Lemma~\ref{lem.lower lipschitz optimal cases}, this bound is optimal for several choices of $G$.
It would be interesting to determine a lower Lipschitz bound in cases where $G$ is infinite.
However, such a bound is also an open problem in the same complex phase retrieval setting where $G:=\{\omega I_k:\omega\in\mathbb{C},|\omega|=1\}$~\cite{BandeiraCMN:14}.
Another open problem in this neighborhood is Problem~19 in~\cite{CahillIMP:22}, which asks whether an injective max filter bank is necessarily bilipschitz.
We note that while explicit lower Lipschitz bounds are not known in the setting of complex phase retrieval, a compactness argument gives that injectivity implies bilipschitz; see Proposition~1.4 in~\cite{CahillCD:16}.

In Theorems~\ref{thm.generic bilipschitz} and~\ref{thm.distortion from random templates}, we estimated how many generic templates are needed for a max filter bank to be bilipschitz, and we estimated the distortion of max filter banks from Gaussian templates.
Both estimates are given in terms of the Voronoi characteristic $\chi(G)$.
In Lemma~\ref{lem.reflection group characterization}, we established that $\chi(G)=1$ precisely when $G$ is a finite reflection group.
For which $G$ does it hold that $\chi(G)=2$?
More generally, it would be interesting to develop tools to estimate the Voronoi characteristic of any given group.
Unfortunately, linear programming experiments suggest that $|S(x,y)|$ is not always essentially constant (e.g., when $G$ is the symmetry group of the regular tetrahedron in $\mathbb{R}^3$).
Finally, Lemma~\ref{lem.reflection group characterization} identifies a simple relationship between $|C([x],[y])|$ and $\chi(G)$, and it would be interesting to establish a deeper relationship between these quantities.

In Section~\ref{sec.pos def}, we detected whether max filtering is positive definite according to whether the orbit space is Euclidean.
For this, we used existing results on the geometry of isometric group actions~\cite{Dadok:85,Michor:97,Grove:02,AlekseevskyKLM:03,Meinrenken:03,DiazRamos:08,Michiels:14,AlexandrinoB:15} and on the polarity and geometry of orthogonal Lie group representations~\cite{LytchakT:10,GordoskiL:14,GordoskiL:16}.
This literature might offer a gateway to even deeper results about max filtering.
For example, it would be interesting to establish a connection between max filtering and infinitesimal polarity~\cite{GordoskiL:16}.

\section*{Acknowledgments}

DGM was supported by NSF DMS 2220304. The authors thank Radu Balan, Ben Blum-Smith, Jameson Cahill, Joey Iverson, Tom Needham, Dan Packer, Clay Shonkwiler, and Soledad Villar for enlightening discussions. The authors thank the anonymous referees for various suggestions that substantially improved the presentation of our results.

\appendix

\section{Proof of Lemma~\ref{lem.characterization of P(G)}}
\label{sec.proof of principal points}

(a)$\Rightarrow$(b).
Select $x\in\mathbb{R}^d$, $y\in V_x$, and $g\in G$.
Then \eqref{eq.open voronoi characterization} gives
\[
\arg\min_{q\in[gx]}\|q-gy\|
=\arg\min_{q\in[x]}\|g^{-1}q-y\|
=\arg\min_{gr\in[x]}\|r-y\|
=g\cdot\arg\min_{r\in[x]}\|r-y\|
=\{gx\},
\]
which in turn implies $gy\in V_{gx}$.
In the case where $x\in P(G)$, we have $|G|$ different Voronoi cells $\{V_{gx}\}_{g\in G}$, each containing a different member of $[y]$.
It follows that $G$ acts freely and transitively on $\{V_p\}_{p\in[x]}$.

(b)$\Rightarrow$(c).
Suppose $G$ acts freely and transitively on $\{V_p\}_{p\in[x]}$, and select $y\in V_x$.
Then every member of $[y]\setminus\{y\}$ resides in some $V_p\subseteq(\overline{V_x})^c$, and so 
\[
\{y\}
=[y]\cap\overline{V_x}
=\arg\min_{p\in[y]}\|p-x\|.
\]
Then $x\in V_y$ by~\eqref{eq.open voronoi characterization}.

(c)$\Rightarrow$(a).
Suppose $x\not\in P(G)$.
Then there exists $g\neq\operatorname{id}$ such that $gx=x$.
Since $V_x$ is open and $P(G)$ is dense, we may select $y\in V_x\cap P(G)$.
Since $y\in V_x\subseteq\overline{V_x}$, the symmetric condition \eqref{eq.closed voronoi characterization} implies $x\in\overline{V_y}$, and so $x=gx\in\overline{V_{gy}}$.
Since $y\in P(G)$, we have $gy\neq y$, meaning $\overline{V_{gy}}\subseteq V_y^c$, and so $x\in V_y^c$.

\section{Proof of Lemma~\ref{lem.equivalence Sxy}}
\label{sec.proof of Sxy}

(a)$\Rightarrow$(b).
Select $q\in[y]\setminus T$.
Since $S(x,y)\subseteq T$, it follows that $V_q\cap V_x=\varnothing$.
Thus, $V_x^c\supseteq V_q$, and since $V_x^c$ is closed, we have $V_x^c\supseteq \overline{V_q}$, meaning $\overline{V_q}\cap V_x=\varnothing$.
As such,
\[
V_x
\subseteq\mathbb{R}^d\setminus\bigg(\bigcup_{q\in[y]\setminus T}\overline{V_q}\bigg)
\subseteq\bigcup_{p\in T}\overline{V_p}.
\]
The result now follows from the fact that the right-hand side is closed.

(b)$\Rightarrow$(a).
Select $q\in [y]\setminus T$.
Then our assumption on $T$ implies
\[
V_x
\subseteq\overline{V_x}
\subseteq\bigcup_{p\in T}\overline{V_p}
\subseteq\bigcup_{p\in [y]\setminus\{q\}}\overline{V_p}
=V_q^c.
\]
As such, $V_q\cap V_x=\varnothing$, and so $q\not\in S(x,y)$.

(b)$\Rightarrow$(c).
Take $v\in\arg\max_{p\in[z]}\langle p,x\rangle$.
Then $v\in\overline{V_x}$, and so by assumption, there exists $w\in T$ such that $v\in\overline{V_w}$, meaning $w\in\arg\max_{q\in[y]}\langle q,v\rangle$.
Then $w$ witnesses the desired nonemptiness.

(c)$\Rightarrow$(b).
Suppose $x\in P(G)$ and $z\in V_x$.
Then Lemma~\ref{lem.characterization of P(G)} gives that $x\in V_z$.
By the characterization~\eqref{eq.open voronoi characterization} of $V_z$, it follows that $\{z\}=\arg\max_{p\in[z]}\langle x,p\rangle$.
By assumption, there exists $v\in\{z\}$ such that $T\cap\arg\max_{q\in[y]}\langle q,v\rangle\neq\varnothing$.
That is, there exists $p\in T$ such that $z\in\overline{V_p}$.
This shows that $V_x\subseteq\bigcup_{p\in T}\overline{V_p}$, and so we are done by taking closures.

\section{Proof of Lemma~\ref{lem.lower lipschitz optimal cases}}
\label{sec.proof of optimal lower lipschitz}

\subsection*{Lemma~\ref{lem.lower lipschitz optimal cases}(c)}

First, $P(G)$ is the union $W$ of open Weyl chambers of $G$, and for each $z\in \mathbb{R}^d$, it holds that $W\subseteq Q_z$.
Thus, $\mathcal{O}(\{z_i\}_{i=1}^n,G)=W$.
For each $x\in W$, $v_i(x)$ is the member of $[z_i]$ that resides in the closed Weyl chamber $\overline{V_x}$.
Similarly, for each $x\in W$ and $y\in\mathbb{R}^d$, the unique element of $S(x,y)$ is the member of $[y]$ that resides in $\overline{V_x}$, and so the only function in $\mathcal{F}(x,y)$ maps each $i\in\{1,\ldots,n\}$ to this member of $[y]$.
It follows that
\[
\max_{f\in\mathcal{F}(x,y)}\sum_{w\in S(x,y)}\lambda_{\min}\bigg(\sum_{i\in f^{-1}(w)} v_i(x)v_i(x)^\top\bigg)
=\lambda_{\min}\bigg(\sum_{i=1}^n v_i(x)v_i(x)^\top\bigg).
\]
Considering $\{v_i(x)\}_{i=1}^n$ is obtained by applying $G$ to send each $z_i$ into $x$'s closed Weyl chamber, and considering $G$ acts transitively on the Weyl chambers, it follows that the right-hand side above is constant for $x\in W$, and $\alpha(\{z_i\}_{i=1}^n,G)$ is the square root of this constant.
The result follows by taking any $x\in W$, selecting a bottom eigenvector $v$ of $\sum_{i=1}^n v_i(x)v_i(x)^\top$, and then putting $y:=x+tv$ with $t$ small enough so that $y\in V_x$.

\subsection*{Lemma~\ref{lem.lower lipschitz optimal cases}(b)}
First, $P(G)=\mathbb{R}^d\setminus\{0\}$, and
\[
Q_z
=\left\{\begin{array}{cl}
\mathbb{R}^d&\text{if }z=0\\
\mathbb{R}^d\setminus\operatorname{span}\{z\}^\perp&\text{if }z\neq0.
\end{array}\right.
\]
For each $x\in\mathcal{O}(\{z_i\}_{i=1}^n,G)$, we have $v_i(x)
=\operatorname{sign}(\langle z_i,x\rangle) z_i$.
Also, for each $x,y\in\mathcal{O}(\{z_i\}_{i=1}^n,G)$, we have $S(x,y)=\{\pm y\}\setminus\operatorname{cone}\{-x\}$.
Since $v_i(x)v_i(x)^\top=z_iz_i^\top$ and $|S(x,y)|\leq 2$ regardless of $x,y\in\mathcal{O}(\{z_i\}_{i=1}^n,G)$, we have
\[
\alpha(\{z_i\}_{i=1}^n,G)
\geq\min_{I\sqcup J=[n]}\bigg(\lambda_{\min}\bigg(\sum_{i\in I}z_iz_i^\top\bigg)+\lambda_{\min}\bigg(\sum_{i\in J}z_iz_i^\top\bigg)\bigg)^{1/2}
=:\Delta.
\]
Following~\cite{BalanW:15}, we select $I$ and $J$ that minimize the right-hand side and take bottom eigenvectors $u$ of $\sum_{i\in I}z_iz_i^\top$ and $v$ of $\sum_{i\in J}z_iz_i^\top$, both of unit norm.
Selecting $x=u+v$ and $y=u-v$ then gives $\|x-y\|=2=\|x+y\|$, and so $d([x],[y])=2$.
An application of the reverse triangle inequality then gives
\begin{align*}
\|\Phi([x])-\Phi([y])\|^2
&=\sum_{i\in I}\Big(|\langle z_i,x\rangle|-|\langle z_i,y\rangle|\Big)^2+\sum_{i\in J}\Big(|\langle z_i,x\rangle|-|\langle z_i,y\rangle|\Big)^2\\
&\leq\sum_{i\in I}|\langle z_i,x+y\rangle|^2+\sum_{i\in J}|\langle z_i,x-y\rangle|^2\\
&=4\Delta^2\\
&\leq \alpha(\{z_i\}_{i=1}^n,G)^2\cdot d([x],[y])^2.
\end{align*}
Combining with Theorem~\ref{thm.lower lipschitz bound} then gives equality in the above inequality.

\subsection*{Lemma~\ref{lem.lower lipschitz optimal cases}(a)}
First, we recall a theorem attributed to Leonardo da Vinci~\cite{Weyl:52} that $G$ is necessarily dihedral or cyclic.
The dihedral case is treated in (c), and the case where $G$ is cyclic and $|G|=2$ is treated in (b).
It remains to consider the case where $G$ is cyclic and $|G|\geq3$.

Since the proof is long, we divide it into steps and we give helpful insights at the end of every step. Figure~\ref{fig:Sxy illustration} could aid in visualization.\\

\vspace{-0.3cm}
\noindent \textbf{Step 1: Duality with singular value maximization.}

We take all the $z_i$'s to be nonzero without loss of generality.
Then $z_i\in P(G)$ for all $i$.
We claim that this implies that $\mathcal{F}(x,y)$ is a singleton set for all $x,y\in\mathcal{O}(\{z_i\}_{i=1}^n,G)$.
To see this, it suffices to show that $\arg\max_{q\in[y]}\langle q,v_i(x)\rangle$ is a singleton set for each $i$, which by \eqref{eq.disjoint union of cones} is equivalent to having $v_i(x)\in Q_y$ for each $i$.
Since $Q_y$ is $G$-invariant, this in turn is equivalent to having $z_i\in Q_y$ for each $i$.
Fix $i$ and recall that $y\in\mathcal{O}(\{z_i\}_{i=1}^n,G)\subseteq Q_{z_i}$, i.e., $y\in V_{gz_i}$ for some $g\in G$.
Then by Lemma~\ref{lem.characterization of P(G)}, it follows that $gz_i\in V_y$, i.e., $z_i\in Q_y$.

Write $G=\langle[\begin{smallmatrix}\cos\theta&-\sin\theta\\\sin\theta&\phantom{-}\cos\theta\end{smallmatrix}]\rangle$ for $\theta=\frac{2\pi}{m}$ with $m\geq3$ and observe that $d=2$ implies
\[
\lambda_{\min}\bigg(\sum_{i} v_i(x)v_i(x)^\top\bigg)
+\lambda_{\max}\bigg(\sum_{i} v_i(x)v_i(x)^\top\bigg)
=\operatorname{tr}\bigg(\sum_{i} v_i(x)v_i(x)^\top\bigg)
=\sum_{i}\|z_i\|^2.
\]
Since $|\mathcal{F}(x,y)|=1$ for all $x,y\in\mathcal{O}(\{z_i\}_{i=1}^n,G)$, we have
\begin{equation}
\label{eq.sup reformulation}
\sup_{x,y\in\mathcal{O}(\{z_i\}_{i=1}^n,G)}\sum_{w\in S(x,y)}\lambda_{\max}\bigg(\sum_{v_i(x)\in \overline{V_w}} v_i(x)v_i(x)^\top\bigg)
=\sum_{i=1}^n\|z_i\|^2-\alpha(\{z_i\}_{i=1}^n,G)^2.
\end{equation}

By virtue of this equation and up until the last step of the proof, we shall base our analysis on the optimizers the left-hand side of~\eqref{eq.sup reformulation}.\\

\vspace{-0.3cm}
\noindent \textbf{Step 2: Notational Setup.}

Each $x\in\mathcal{O}(\{z_i\}_{i=1}^n,G)$ determines $\{v_i(x)\}_{i=1}^n\in [z_1]\times\cdots\times[z_n]$, and each $x,y\in\mathcal{O}(\{z_i\}_{i=1}^n,G)$ determines a partition of $\{1,\ldots,n\}$ into subsets $\{A_w\}_{w\in S(x,y)}$ of the form
\[
A_w:=\{i:v_i(x)\in \overline{V_w}\}.
\]
Since there are finitely many choices of $\{v_i(x)\}_{i=1}^n$ and $\{A_w\}_{w\in S(x,y)}$, there necessarily exist $x,y\in\mathcal{O}(\{z_i\}_{i=1}^n,G)$ that maximize the left-hand side of \eqref{eq.sup reformulation}.
Fix such $x^\star$ and $y^\star$.
Without loss of generality, we change coordinates so that $x^\star$ has argument $\theta/2$, where $\theta:=2\pi/|G|$.
Then $0<\arg v_i(x^\star)<\theta$ for every $i\in\{1,\ldots,n\}$.
For each $y\in \mathcal{O}(\{z_i\}_{i=1}^n,G)$, there exists $\tau\in[0,\theta]$ such that $\{A_w\}_{w\in S(x^\star,y)}$ partitions indices $i\in\{1,\ldots,n\}$ according to whether $\arg v_i(x^\star)$ is strictly less than or strictly greater than $\tau$.

If $\{v_i(x^\star)\}_{i=1}^n$ are all positive multiples of each other, then $\alpha(\{z_i\}_{i=1}^n,G)=0$.
In this case, we may take $x=v_1(x^\star)$ and $y=v_1(x^\star)+w$ for some small enough $w\neq 0$ orthogonal to $\{v_i(x^\star)\}_{i=1}^n$ so that $y$ resides in their common Voronoi cell.
Then $\|\Phi([x])-\Phi([y])\|=0$.
Otherwise, $\{v_i(x^\star)\}_{i=1}^n$ contains a pair of linearly independent vectors.
Select $y$ so that the corresponding threshold $\tau$ splits the arguments of these two vectors.
Then
\begin{align*}
&\lambda_{\max}\bigg(\sum_{i=1}^n v_i(x^\star)v_i(x^\star)^\top\bigg)\\
&\qquad<\lambda_{\max}\bigg(\sum_{\arg v_i(x^\star)<\tau} v_i(x^\star)v_i(x^\star)^\top\bigg)
+\lambda_{\max}\bigg(\sum_{\arg v_i(x^\star)>\tau} v_i(x^\star)v_i(x^\star)^\top\bigg)\\
&\qquad\leq \sum_{w\in S(x^\star,y^\star)}\lambda_{\max}\bigg(\sum_{i\in A_w} v_i(x^\star)v_i(x^\star)^\top\bigg).
\end{align*}
As such, the optimal partition $\{A_w\}_{w\in S(x^\star,y^\star)}$ consists of two nonempty sets, which we denote by $A$ and $B$ for simplicity.
Let $\tau^\star\in(0,\theta)$ denote a corresponding threshold.
We denote the minimum and maximum arguments of $\{v_i(x^\star)\}_{i\in A}$ by $\phi_1$ and $\phi_2$, respectively.
We similarly denote the extreme arguments of $\{v_i(x^\star)\}_{i\in B}$ by $\psi_1$ and $\psi_2$:
\[
0
<\phi_1
<\phi_2
<\tau^\star
<\psi_1
<\psi_2
<\theta.
\]

\begin{figure}[t]
  \centering
  \includegraphics[scale = 1.25]{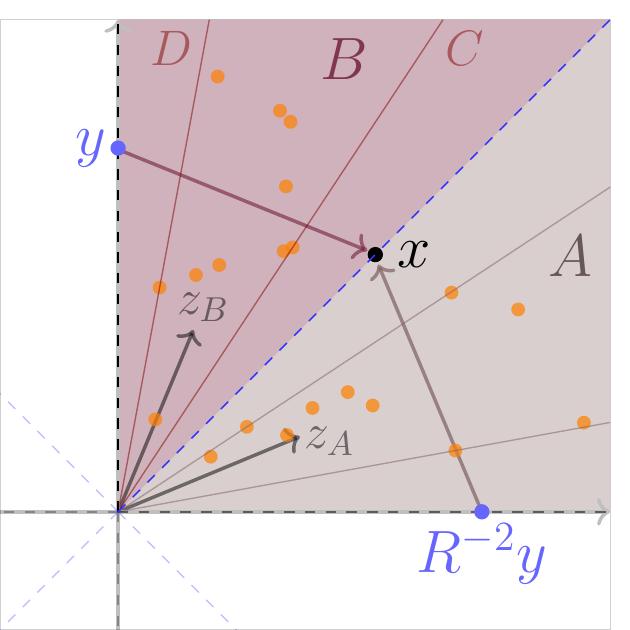}
  \caption{Illustration for the optimal partition in the proof of Lemma~\ref{lem.lower lipschitz optimal cases}(a). Here, $G\leq \operatorname O(2)$ consists of rotations by multiples of $\frac{\pi}{2}$ radians and the figure zooms in on the Voronoi cell $V_x$ given by the open first quadrant of the plane. The orange points in the portion of $V_x$ below (resp. above) the diagonal comprise $\{v_i(x)\}_{i\in A}$ (resp.\ $\{v_i(x)\}_{i\in B}$). The vectors $z_A$ and $z_B$ are top eigenvectors corresponding to $\sum_{i\in A} v_i(x)v_i(x)^\top$ and $\sum_{i\in B} v_i(x)v_i(x)^\top$ respectively. Moreover, $x = z_A + z_B$ and $y = tRx$, where $R$ is the counterclockwise rotation by $\frac{\pi}{4}$ which forces $d([x],[y]) = \|x-y\|=\|x-R^{-2}y\|$, and $t>0$ is chosen in such a way that $x-R^{-2}y$ is orthogonal to $z_A$ and $x-y$ is orthogonal $z_B$; hence, each is parallel to the bottom eigenvectors of $\sum_{i\in A} v_i(x)v_i(x)^\top$ and $\sum_{i\in B} v_i(x)v_i(x)^\top$ respectively.
  }
  \label{fig:upper aid}
\end{figure}

\noindent\textbf{Step 3. Acute Cones Claim.}

We claim that both $\phi_2-\phi_1<\pi/2$ and $\psi_2-\psi_1<\pi/2$.
If $m\geq4$, then this follows immediately from the fact that $\theta\leq\pi/2$.
When $m=3$, pigeonhole gives that at least one of $\phi_2-\phi_1<\pi/2$ or $\psi_2-\psi_1<\pi/2$ holds.
Without loss of generality, we assume $\phi_2-\phi_1<\pi/2$.
To prove $\psi_2-\psi_1<\pi/2$, we analyze the arguments of the top left singular vectors of $\{v_i(x^\star)\}_{i\in A}$ and $\{v_i(x^\star)\}_{i\in B}$ (interpreted as $2\times|A|$ and $2\times|B|$ matrices).
In particular, select unit vectors $z_A,z_B\in\mathbb{R}^2$ that maximize the quantities
\begin{equation}
\label{eq.top sing vecs}
\sum_{i\in A}|\langle v_i(x^\star),z_A\rangle|^2,
\qquad
\sum_{i\in B}|\langle v_i(x^\star),z_B\rangle|^2.
\end{equation}
Notice that the maximizers are unique up to a sign.
Since $\phi_2-\phi_1<\pi/2$ by assumption, $\{v_i(x^\star)\}_{i\in A}$ generates an acute cone, and so we may select $z_A$ so as to reside in this cone.
To see this, note that the Gram matrix of $\{v_i(x^\star)\}_{i\in A}$ is entrywise positive, and so Perron--Frobenius gives that it has an entrywise positive top eigenvector; multiplying by the matrix $\{v_i(x^\star)\}_{i\in A}$ and normalizing gives a top left singular vector in the interior of the cone that they generate (see the proof of Lemma~17(a) in~\cite{MixonP:22} for additional details).
We also select $z_B$ to have argument in $[0,\pi)$.
Let $\gamma_A$ and $\gamma_B$ denote the arguments of $z_A$ and $z_B$, respectively.
Then we have
\[
\phi_1<\gamma_A<\phi_2,
\qquad
0\leq\gamma_B<\pi.
\]
Since $\{A,B\}$ is the optimal partition of $\{1,\ldots,n\}$ that is feasible for our problem, we may isolate two ways in which it is locally optimal.
First, consider the subset $C\subseteq B$ such that $v_i(x^\star)$ has minimum argument $\psi_1$ for all $i\in C$.
Then an appropriate perturbation $y\in\mathcal{O}(\{z_i\}_{i=1}^n,G)$ of $y^\star$ gives
\[
\sum_{i\in A}|\langle v_i(x^\star),z_A\rangle|^2+\sum_{i\in B}|\langle v_i(x^\star),z_B\rangle|^2
\geq\sum_{i\in A\cup C}|\langle v_i(x^\star),z_A\rangle|^2+\sum_{i\in B\setminus C}|\langle v_i(x^\star),z_B\rangle|^2.
\]
It follows that $\cos^2(\psi_1-\gamma_B)\geq\cos^2(\psi_1-\gamma_A)$, which in turn implies
\[
\min_{k\in\{-1,0,1\}}|\psi_1-\gamma_B-k\pi|
=\min_{k\in\mathbb{Z}}|\psi_1-\gamma_B-k\pi|
\leq\min_{k\in\mathbb{Z}}|\psi_1-\gamma_A-k\pi|
\leq\psi_1-\gamma_A.
\]
Next, consider the subset $D\subseteq B$ such that $v_i(x^\star)$ has maximum argument $\psi_2$ for all $i\in D$.
Then an appropriate perturbation $x\in\mathcal{O}(\{z_i\}_{i=1}^n,G)$ of $x^\star$ gives
\begin{align*}
&\sum_{i\in A}|\langle v_i(x^\star),z_A\rangle|^2+\sum_{i\in B}|\langle v_i(x^\star),z_B\rangle|^2\\
&\qquad\geq\sum_{i\in A\cup D}|\langle v_i(x),z_A\rangle|^2+\sum_{i\in B\setminus D}|\langle v_i(x),z_B\rangle|^2\\
&\qquad=\sum_{i\in A}|\langle v_i(x^\star),z_A\rangle|^2+\sum_{i\in D}|\langle [\begin{smallmatrix}\cos(-\theta)&-\sin(-\theta)\\\sin(-\theta)&\phantom{-}\cos(-\theta)\end{smallmatrix}]v_i(x^\star),z_A\rangle|^2+\sum_{i\in B\setminus D}|\langle v_i(x^\star),z_B\rangle|^2.
\end{align*}
It follows that $\cos^2(\psi_2-\gamma_B)\geq\cos^2(\psi_2-\theta-\gamma_A)$, which in turn implies
\[
\min_{\ell\in\{-1,0,1\}}|\psi_2-\gamma_B-\ell\pi|
=\min_{\ell\in\mathbb{Z}}|\psi_2-\gamma_B-\ell\pi|
\leq\min_{\ell\in\mathbb{Z}}|\psi_2-\theta-\gamma_A-\ell\pi|
\leq-(\psi_2-\theta-\gamma_A).
\]
For each $k,\ell\in\{-1,0,1\}$, one may run Fourier--Motzkin to eliminate $\gamma_B$ from the system
\[
|\psi_1-\gamma_B-k\pi|
\leq\psi_1-\gamma_A,
\qquad
|\psi_2-\gamma_B-\ell\pi|
\leq-(\psi_2-\theta-\gamma_A).
\]
In each case, this either reveals that the above system is infeasible, or that the system implies $\psi_1-\psi_2\leq\theta/2=\pi/3$.
Our intermediate claim follows.\\

\vspace{-0.3cm}
\noindent\textbf{Step 4. The bisector of top left singular vectors lies in between the cones.}

Now that we have both $\phi_2-\phi_1<\pi/2$ and $\psi_2-\psi_1<\pi/2$, then $\{v_i(x^\star)\}_{i\in A}$ and $\{v_i(x^\star)\}_{i\in B}$ both generate acute cones, and so we may select maximizers $z_A$ and $z_B$ of \eqref{eq.top sing vecs} so that the corresponding arguments $\gamma_A$ and $\gamma_B$ satisfy
\[
0
<\phi_1
<\gamma_A
<\phi_2
<\tau^\star
<\psi_1
<\gamma_B
<\psi_2
<\theta.
\]
Since $\{A,B\}$ is the optimal partition of $\{1,\ldots,n\}$, we may again isolate two ways in which it is locally optimal (just as we did above):
\[
-(\psi_1-\gamma_B)
=\min_{k\in\mathbb{Z}}|\psi_1-\gamma_B-k\pi|
\leq\min_{k\in\mathbb{Z}}|\psi_1-\gamma_A-k\pi|
\leq\psi_1-\gamma_A,
\]
\[
\psi_2-\gamma_B
=\min_{\ell\in\mathbb{Z}}|\psi_2-\gamma_B-\ell\pi|
\leq\min_{\ell\in\mathbb{Z}}|\psi_2-\theta-\gamma_A-\ell\pi|
\leq-(\psi_2-\theta-\gamma_A).
\]
The above inequalities are obtained by passing extreme points from $B$ to $A$.
We may also pass extreme points from $A$ to $B$ to obtain additional inequalities:
\[
\phi_2-\gamma_A
=\min_{k\in\mathbb{Z}}|\phi_2-\gamma_A-k\pi|
\leq\min_{k\in\mathbb{Z}}|\phi_2-\gamma_B-k\pi|
\leq-(\phi_2-\gamma_B)
\]
\[
-(\phi_1-\gamma_A)
=\min_{\ell\in\mathbb{Z}}|\phi_1-\gamma_A-\ell\pi|
\leq\min_{\ell\in\mathbb{Z}}|\phi_1+\theta-\gamma_B-\ell\pi|
\leq\phi_1+\theta-\gamma_B.
\]
Rearranging these inequalities then gives
\begin{equation}
\label{eq.fourier motzkin results}
\frac{\gamma_A+\gamma_B}{2}-\frac{\theta}{2}
\leq \phi_1
< \phi_2
\leq \frac{\gamma_A+\gamma_B}{2}
\leq \psi_1
< \psi_2
\leq \frac{\gamma_A+\gamma_B}{2}+\frac{\theta}{2}.
\end{equation}

The inner inequalities imply that the the bisector $z_A+z_B$ of top left singular vectors $z_A$ and $z_B$ lies in between the cones generated by $A$ and $B$. The outer inequalities imply that both cones are contained in the Voronoi cell of that bisector.\\

\vspace{-0.3cm}
\noindent \textbf{Step 5. Appropriate choice of $x$ and $y$.}

Observe that the bisector $x:=z_A+z_B$ has argument $(\gamma_A+\gamma_B)/2$. Let $R$ denote counterclockwise rotation by $\theta/2$ radians, and put $y:=tRx$, where $t:=\langle x,z_B\rangle/\langle Rx,z_B\rangle$.
Note that $t>0$ since argument $\gamma_B$ of $z_B$ is within $\theta/2$ of the arguments of $x$ and $Rx$ by \eqref{eq.fourier motzkin results}. Geometrically speaking, $y$ lies on the boundary of $V_x$ and the of choice of $t$ implies that $y-x$ is orthogonal to $z_B$ and $x-R^{-2}y$ is orthogonal to $z_A$. The latter orthogonality geometrically follows from the first from the fact that $x$ is simultaneously a bisector of $z_A$ and $z_B$ and a bisector of $y$ and $R^{-2}y$.

As such, $x-R^{-2}y$ and $y-x$ are scalar multiples of the bottom eigenvectors of $\sum_{i\in A} v_i(x^\star)v_i(x^\star)^\top$ and $\sum_{i\in B} v_i(x^\star)v_i(x^\star)^\top$, respectively. This is crucially due to the fact that $d=2$.
Note that by \eqref{eq.fourier motzkin results}, $\llangle [z_i],[x]\rrangle=\langle v_i(x^\star),x\rangle$ for all $i\in\{1,\ldots,n\}$ and
\[
\llangle [z_i],[y]\rrangle
=\left\{\begin{array}{cl}
\langle v_i(x^\star),y\rangle&\text{if } i\in B\\
\langle v_i(x^\star),R^{-2}y\rangle&\text{if } i\in A.
\end{array}\right.
\]
Furthermore, since $y\in\operatorname{cone}\{Rx\}$, it holds that $y,R^{-2}y\in \overline V_x$ and so $d([x],[y])=\|x-y\|=\|x-R^{-2}y\|$.
As such,
\begin{align*}
&\|\Phi([x])-\Phi([y])\|^2\\
&\qquad=\sum_{i\in A}\langle v_i(x^\star),x-R^{-2}y\rangle^2+\sum_{i\in B}\langle v_i(x^\star),x-y\rangle^2\\
&\qquad=\lambda_{\min}\bigg(\sum_{i\in A} v_i(x^\star)v_i(x^\star)^\top\bigg)\cdot\|x-R^{-2}y\|^2+\lambda_{\min}\bigg(\sum_{i\in B} v_i(x^\star)v_i(x^\star)^\top\bigg)\cdot\|x-y\|^2\\
&\qquad=\alpha(\{z_i\}_{i=1}^n,G)^2\cdot d([x],[y])^2,
\end{align*}
as desired.

\section{Details for the proof of Theorem~\ref{thm.distortion from random templates}}\label{sec.proof of distortion from random templates}

For~\eqref{eq.random upper bound}, first note that $d\sqrt \lambda \leq d\chi\lambda = n \leq n\log(em)$. We obtain
\[\sqrt d + \sqrt n + \sqrt{2(n\log m + d\sqrt \lambda)} \leq \left(1+1+\sqrt{4}\right)\sqrt{n\log (em)} = 4 \sqrt{n \log (em)}.\] 
For~\eqref{eq.random lower sigma bound}, Theorem~1.3 in~\cite{Wang:18} (with $\mu$ taken to be $t\lambda$) gives that
\[\sigma_{\min}(X) \geq \left(\frac{1}{2C\lambda L^\lambda}\right)^{\frac{1}{\lambda-1}}\left(1-\frac{1}{\lambda}\right)\sqrt{\frac{l}{\lambda}},\]
with probability at least $1-2e^{-tl}$ where $C = 1+\sqrt \lambda + \sqrt{2t\lambda}$ and $L=\sqrt{\frac{e}{\lambda}}e^{t}$. By plugging in the formula for $L$ and rearranging, we obtain
\[\sigma_{\min}(X)\geq \frac{1}{\sqrt e}\left(1-\frac{1}{\lambda}\right)\left(\frac{\lambda^{\lambda/2}\lambda^{\frac{-(\lambda-1)}{2}}}{2C\sqrt{e}\lambda}\right)^{\frac{1}{\lambda-1}}e^{-t\lambda/(\lambda-1)}\sqrt l.\]
To arrive to~\eqref{eq.random lower sigma bound}, it suffices to show $\frac{\lambda^{1/2}}{C}\geq \frac{1}{2+\sqrt 2}\cdot \frac{1}{\sqrt t}$. By plugging in the definition of $C$ and rearranging, this is equivalent to $\frac{1}{\sqrt \lambda} + 1 + \sqrt{2t}\leq (2+\sqrt 2)\sqrt t$ which holds true whenever $\lambda >1$ and $t\geq 1$.

For~\eqref{eq.random lower proba bound I}, the aim is to establish that
\[{n\choose \lceil n/\chi \rceil} \cdot m^{\lceil n/\chi \rceil} \cdot e^{-\lceil n/\chi \rceil t} \leq \exp\left(\left(\frac{n}{\chi}+1\right)\log(e\chi m)-\frac{nt}{\chi}\right).\]
Using the bound ${n\choose \lceil n/\chi \rceil}\leq \left(\frac{n\cdot e}{\lceil n/\chi \rceil}\right)^{\lceil n/\chi \rceil}$ and taking logs on both sides, it suffices to show that
\[\lceil n/\chi \rceil \log\left(\frac{n\cdot e}{\lceil n/\chi \rceil}\right) + \lceil n/\chi \rceil \log(m)-\lceil n/\chi \rceil t\leq\left(\frac{n}{\chi}+1\right)\log(e\chi m)-\frac{nt}{\chi},\]
which holds since $\lceil n/\chi \rceil\leq \frac{n}{\chi}+1$, $\frac{nem}{\lceil n/\chi \rceil} \leq e\chi m$ and $\lceil n/\chi \rceil t \geq \frac{nt}{\chi}$.

Next, to establish~\eqref{eq.random lower proba bound II}, we plug in $t=\left(1+\frac{2}{\sqrt \lambda}\right)\log(e\chi m)$ and $n/\chi = d\lambda$ into \eqref{eq.random lower proba bound I} to aim for establishing
\[(d\lambda +1)\log(e\chi m) - d\lambda\left(1+\frac{2}{\sqrt \lambda}\right)\log(e\chi m) \leq -d\sqrt \lambda.\]
By rearrangement, this is equivalent to $\left(\frac{1}{d\sqrt \lambda }-2\right)\log(e\chi m) \leq -1$ which holds true since $\log(e\chi m)\geq 1$ and $d\sqrt \lambda \geq 1$.

Lastly, we aim to show that~\eqref{eq.random distortion bound} yields the conclusion of the theorem. We aim for an upper bound of
\begin{align*}
H&:= \frac{4\sqrt{n\log (em)}}{\sigma\left(\lceil n/\chi \rceil,\lambda,\left(1+\frac{2}{\sqrt \lambda}\right)\log(e\chi m)\right)}\\
&= \sqrt e \cdot 4\sqrt{n\log (em)}\sqrt{\lceil n/\chi \rceil^{-1}}\cdot \big(2(2 + \sqrt 2)\sqrt e\big)^{\frac{1}{\lambda-1}} \cdot \lambda^{\frac{1}{\lambda-1}}\cdot\\
& \qquad \qquad \frac{\lambda}{\lambda-1}\cdot \big({1+2/\sqrt \lambda}\big)^{\frac{0.5}{\lambda-1}} \log^{0.5}(e\chi m)^{\frac{1}{\lambda-1}}\cdot (e\chi m)^{\left(1+\frac{2}{\sqrt \lambda}\right)\frac{\lambda}{\lambda - 1}}.
\end{align*}

By noting the power in the very last term, we aim for a constant $c(\lambda_0)$ such that $\left(1+\frac{2}{\sqrt \lambda}\right)\frac{\lambda}{\lambda - 1} \leq 1+ c/\sqrt{\lambda}$, which is equivalent to $c\geq 2 + \frac{\sqrt \lambda +2}{\lambda-1}$. Since $\lambda \geq \lambda_0 > 1$ and $\frac{\sqrt \lambda +2}{\lambda-1}$ is decreasing in $\lambda > 1$, we may take $c(\lambda_0) := 2 + \frac{\sqrt \lambda_0 +2}{\lambda_0-1}$. With $\sqrt{n\log (em)}\sqrt{\lceil n/\chi \rceil^{-1}} \leq \sqrt{\chi \log(em)}$, we get
\[H\leq \left[\frac{\lambda^{1+\frac{1}{\lambda-1}}(2(2+\sqrt 2)\sqrt e)^{\frac{1}{\lambda-1}}}{\lambda-1}\cdot ({1+2/\sqrt \lambda})^{\frac{0.5}{\lambda-1}}\cdot 4\sqrt e\right] \cdot \left(e \chi^{3/2}m\log^{1/2}(em)\right)^{1+c/\sqrt \lambda}.\]
Since
\[\left[\frac{\lambda^{1+\frac{1}{\lambda-1}}(2(2+\sqrt 2)\sqrt e)^{\frac{1}{\lambda-1}}}{\lambda-1}\cdot ({1+2/\sqrt \lambda})^{\frac{0.5}{\lambda-1}}\cdot 4\sqrt e\right]^{\left(\left(1+\frac{2}{\sqrt \lambda}\right)\frac{\lambda}{\lambda - 1}\right)^{-1}}\leq 4\sqrt e,\]
we set $C := 4e^{3/2}$ and the result follows.

\begin{thebibliography}{WW}

\bibitem{AlekseevskyKLM:03}
D.\ Alekseevsky, A.\ Kriegl, M.\ Losik, P.\ W.\ Michor,
The Riemannian geometry of orbit spaces.\ The metric, geodesics, and integrable systems,
Publ.\ Math.\ 62 (2003) 247--276.

\bibitem{AlexandrinoB:15}
M.\ M.\ Alexandrino, R.\ G.\ Bettiol,
Lie groups and geometric aspects of isometric actions,
Springer, 2015.

\bibitem{BalanCE:06}
R.\ Balan, P.\ Casazza, D.\ Edidin,
On signal reconstruction without phase,
Appl.\ Comput.\ Harmon.\ Anal.\ 20 (2006) 345--356.

\bibitem{BalanHS:22}
R.\ Balan, N.\ Haghani, M.\ Singh,
Permutation Invariant Representations with Applications to Graph Deep Learning,
arXiv:2203.07546

\bibitem{BalanW:15}
R.\ Balan, Y.\ Wang,
Invertibility and robustness of phaseless reconstruction,
Appl.\ Comput.\ Harmon.\ Anal.\ 38 (2015) 469--488.

\bibitem{BandeiraBKPWW:17}
A.\ S.\ Bandeira, B.\ Blum-Smith, J.\ Kileel, A.\ Perry, J.\ Weed, A.\ S.\ Wein,
Estimation under group actions:\ Recovering orbits from invariants,
arXiv:1712.10163

\bibitem{BandeiraCMN:14}
A.\ S.\ Bandeira, J.\ Cahill, D.\ G.\ Mixon, A.\ A.\ Nelson,
Saving phase:\ Injectivity and stability for phase retrieval,
Appl.\ Comput.\ Harmon.\ Anal.\ 37 (2014) 106--125.

\bibitem{BendoryELS:22}
T.\ Bendory, D.\ Edidin, W.\ Leeb, N.\ Sharon,
Dihedral multi-reference alignment,
IEEE Trans.\ Inform.\ Theory 68 (2022) 3489--3499.

\bibitem{BergCR:84}
C.\ Berg, J.\ P.\ R.\ Christensen, P.\ Ressel,
Harmonic analysis on semigroups:\ Theory of positive definite and related functions,
Springer, 1984.

\bibitem{BochnakCR:13}
J.\ Bochnak, M.\ Coste, M.-F.\ Roy,
Real algebraic geometry,
Springer, 2013.

\bibitem{CahillCD:16}
J.\ Cahill, P.\ Casazza, I.\ Daubechies,
Phase retrieval in infinite-dimensional Hilbert spaces,
Trans.\ Amer.\ Math.\ Soc., Ser.\ B 3 (2016) 63--76.

\bibitem{CahillCC:20}
J.\ Cahill, A.\ Contreras, A.\ Contreras-Hip,
Complete set of translation invariant measurements with Lipschitz bounds,
Appl.\ Comput.\ Harmon.\ Anal.\ 49 (2020) 521--539.

\bibitem{CahillIM:24}
J.\ Cahill, J.\ W.\ Iverson, D.\ G.\ Mixon,
Towards a bilipschitz invariant theory,
arXiv:2305.17241

\bibitem{CahillIMP:22}
J.\ Cahill, J.\ W.\ Iverson, D.\ G.\ Mixon, D.\ Packer,
Group-invariant max filtering,
arXiv:2205.14039

\bibitem{ChenDL:20}
S.\ Chen, E.\ Dobriban, J.\ Lee,
A group-theoretic framework for data augmentation,
NeurIPS 2020, 21321--21333.

\bibitem{ChoiS:05}
M.-J.\ Choi, D.\ Y.\ Shu,
Comparison of semialgebraic groups with Lie groups and algebraic groups,
RIMS Ky\^{o}ky\^{u}roku 1449 (2005) 12--20.

\bibitem{CiresanMGS:10}
D.\ C.\ Cire\c{s}an, U.\ Meier, L.\ M.\ Gambardella, J.\ Schmidhuber,
Deep, big, simple neural nets for handwritten digit recognition,
Neural Comput.\ 22 (2010) 3207--3220.

\bibitem{ConcaEHV:15}
A.\ Conca, D.\ Edidin, M.\ Hering, C.\ Vinzant,
An algebraic characterization of injectivity in phase retrieval,
Appl.\ Comput.\ Harmon.\ Anal.\ 38 (2015) 346--356.

\bibitem{DiazRamos:08}
J.\ C.\ D\'{i}az-Ramos,
Proper isometric actions,
arXiv:0811.0547

\bibitem{Dadok:85}
J.\ Dadok,
Polar coordinates induced by actions of compact Lie groups,
Trans.\ Amer.\ Math.\ Soc.\ 288 (1985) 125--137.

\bibitem{DudekH:94}
E.\ Dudek, K.\ Holly,
Nonlinear orthogonal projection,
Ann.\ Pol.\ Math.\ 59 (1994) 1--31.

\bibitem{DymG:22}
N.\ Dym, S.\ J.\ Gortler,
Low Dimensional Invariant Embeddings for Universal Geometric Learning,
arXiv:2205.02956 (2022).

\bibitem{GordoskiL:16}
C.\ Gorodski, A.\ Lytchak,
Isometric actions on spheres with an orbifold quotient,
Math.\ Ann.\ 365 (2016) 1041--1067.

\bibitem{GordoskiL:14}
C.\ Gorodski, A.\ Lytchak,
On orbit spaces of representations of compact Lie groups,
J.\ f\"{u}r die Reine und Angew.\ Math.\ 2014 (2014) 61--100.

\bibitem{Grove:02}
K.\ Grove,
Geometry of, and via, symmetries,
Univ.\ Lecture Ser.\ 27 (2002) 31--51.

\bibitem{HeinLB:04}
M.\ Hein, T.\ N.\ Lal, O.\ Bousquet,
Hilbertian metrics on probability measures and their application in SVM's,
in:\ Joint Pattern Recognition Symposium,
Springer, 2004, pp.\ 270--277.

\bibitem{HeinosaariMW:13}
T.\ Heinosaari, L.\ Mazzarella, M.\ M.\ Wolf,
Quantum tomography under prior information,
Comm.\ Math.\ Phys.\ 318 (2013) 355--374.

\bibitem{KrizhevskySH:12}
A.\ Krizhevsky, I.\ Sutskever, G.\ Hinton,
ImageNet Classification with Deep Convolutional Neural Networks,
NeurIPS 2012, 1097--1105.

\bibitem{Lee:06}
J.\ M.\ Lee,
Riemannian manifolds:\ An introduction to curvature,
Springer, 2006.

\bibitem{LytchakT:10}
A.\ Lytchak, G.\ Thorbergsson,
Curvature explosion in quotients and applications,
J.\ Differ.\ Geom.\ 85 (2010) 117--140.

\bibitem{Meinrenken:03}
E.\ Meinrenken,
Group actions on manifolds,
\url{http://www.math.toronto.edu/~mein/teaching/LectureNotes/action.pdf}

\bibitem{Michiels:14}
D.\ Michiels,
Orbit type stratification for proper actions,
2014.

\bibitem{Michor:97}
P.\ W.\ Michor,
Isometric actions of Lie groups and invariants,
Lecture notes, 1997.

\bibitem{MixonP:22}
D.\ G.\ Mixon, D.\ Packer,
Max filtering with reflection groups,
arXiv:2212.05104

\bibitem{Olver:19}
P.\ J.\ Olver,
Invariants of finite and discrete group actions via moving frames,
2019.

\bibitem{OnishchikV:12}
A.\ L.\ Onishchik, E.\ B.\ Vinberg,
Lie groups and algebraic groups,
Springer, 2012.

\bibitem{PerryWBRS:19}
A.\ Perry, J.\ Weed, A.\ S.\ Bandeira, P.\ Rigollet, A.\ Singer, 
The sample complexity of multireference alignment,
SIAM J.\ Math.\ Data Science 1 (2019) 497--517.

\bibitem{SimardSP:03}
P.\ Y.\ Simard, D.\ Steinkraus, J.\ C.\ Platt,
Best practices for convolutional neural networks applied to visual document analysis,
ICDAR 2003, 1--6.

\bibitem{Vershynin:12}
R.\ Vershynin,
Introduction to the non-asymptotic analysis of random matrices,
in:\ Compressed Sensing, Theory and Applications,
Cambridge U.\ Press, 2012, pp.\ 210--268.

\bibitem{Vinzant:15}
C.\ Vinzant,
A small frame and a certificate of its injectivity,
SampTA 2015, 197--200.

\bibitem{Wang:18}
Y.\ Wang,
Random matrices and erasure robust frames,
J.\ Fourier Anal.\ Appl.\ 24 (2018) 1--16.

\bibitem{Weyl:52}
H.\ Weyl,
Symmetry,
Princeton U.\ Press, 1952.

\end{thebibliography}
\end{document}